\documentclass[a4paper,11pt]{amsart}
\usepackage{amsfonts}
\usepackage{amsmath}
\usepackage{amssymb}
\usepackage[a4paper]{geometry}
\usepackage{mathrsfs}
\usepackage{xcolor}
\usepackage{hyperref}
\renewcommand\eqref[1]{(\ref{#1})} 
\usepackage[utf8]{inputenc}
\usepackage[T1]{fontenc}
\newtheorem{remark}{Remark}
\numberwithin{equation}{section}
\theoremstyle{plain}
\newtheorem{theorem}{Theorem}[section]
\newtheorem{lemma}[theorem]{Lemma}
\newtheorem{corollary}{Corollary}

\newtheorem{lem}[theorem]{Lemma}
\newtheorem{proposition}{Proposition}
\theoremstyle{definition}
\newtheorem{definition}[theorem]{Definition}

\begin{document}
\title[local-nonlocal heat equation with exponential nonlinearity]
{Heat equations driven by mixed local-nonlocal operators with exponential nonlinearity  }

\author[D. K. Chaurasia]{Dharmendra Kumar Chaurasia}
\address{
  Dharmendra Kumar Chaurasia:
  \endgraf Department of Mathematics,\endgraf 
		Banaras Hindu University
		Varanasi, \endgraf Uttar Pradesh, 221005, India.
  \endgraf {\it E-mail address:} {\rm  mkc0560@bhu.ac.in}
  }

\author[A. Z. Fino]{Ahmad Z. Fino}
\address{
 Ahmad Z. Fino:
  \endgraf College of Engineering and Technology,\endgraf 
		American University of the Middle East,\endgraf 
		 Egaila 54200, Kuwait.
  \endgraf {\it E-mail address:} {\rm  ahmad.fino@aum.edu.kw}
  }

\author[V. Kumar]{Vishvesh Kumar}
\address{
  Vishvesh Kumar:
  \endgraf Department of Mathematical Sciences,\endgraf 
		Indian Institute of Technology (BHU),
		Varanasi, \endgraf Uttar Pradesh, 221005, India.
  \endgraf {\it E-mail address:} {\rm  vishveshmishra@gmail.com; vishvesh.mat@iitbhu.ac.in}
  }


\thanks{Corresponding author: Vishvesh Kumar, email: vishvesh.mat@iitbhu.ac.in}


\begin{abstract} 
We investigate the Cauchy problem for a heat equation driven by the mixed local-nonlocal operator $\mathcal{L}:=-\Delta+(-\Delta)^s$, $s\in(0,1)$, with exponential nonlinearity
\[
\partial_tu(x,t)+\mathcal{L}u(x,t)=f(u(x,t)),
\qquad (x,t)\in \mathbb{R}^{d}\times(0,\infty),
\]
where $f:\mathbb{R}\to\mathbb{R}$ exhibits exponential growth at infinity and satisfies $f(0)=0$. 
We establish local well-posedness in a suitable Orlicz space in the case where $f(u)\sim e^{|u|^p}$ as $|u|\to\infty$, with $p>1$. 
We further prove the existence of global solutions for small initial data under the assumption that $f$ satisfies the growth condition $|f(u)|\sim |u|^m$ near the origin.  Moreover, we derive large-time decay estimates in Lebesgue spaces, showing that the behavior of the nonlinearity near the origin determines the decay rate of solutions and highlights a unique asymptotic transition that bridges local and non-local diffusion theories.
\end{abstract}
\maketitle
\allowdisplaybreaks
\section{Introduction, main results and discussion}
The main aim of this paper is to study the following local-nonlocal semilinear heat equation: 
\begin{equation}\label{eq1}
\left\{\begin{array}{ll}
\,\, \displaystyle {\partial_tu(x, t)+\mathcal L u(x, t) =f(u(x, t)),} &\displaystyle {t>0,x\in {\mathbb{R}^d},}\\
{}\\
\displaystyle{u(0,x)=  u_0(x),\qquad\qquad}&\displaystyle{x\in {\mathbb{R}^d},}
\end{array}
\right.
\end{equation} 
where $d\geq 1$, $f:\mathbb{R}\rightarrow \mathbb{R}$ is a nonlinear function with exponential growth at infinity ($f(u)\sim e^{|u|^p}$, $p>1$, for large $u$) and satisfying $f(0)=0.$ The operator
$$\mathcal L = -\Delta+(-\Delta)^{s},$$ is the mixed local-nonlocal operator with the operator $(-\Delta)^{s}$ being the (nonlocal) fractional Laplacian of order $s\in (0,1)$ defined by
\begin{align*}
(-\Delta)^{s} v(x) & :=C_{d,s}  \mathrm{P.V.}\int_{\mathbb{R}^d}\frac{v(x)-v(y)}{|x-y|^{d+2s}}\,dy,
\end{align*}
where P.V. stands for the Cauchy principal value, $C_{d,s}:= \frac{s4^{s}\Gamma\left(\frac{d+2s}{2}\right)}{\pi^{d/2}\Gamma(1-s)}$ is a normalization constant, and $\Gamma$ denotes the Gamma function.

It is worth emphasizing that the mixed local-nonlocal operator \(\mathcal{L}\) consists of the superposition of the classical Laplacian \(\Delta\), a local second-order differential operator, and the fractional Laplacian \((-\Delta)^s\), which is inherently nonlocal. While the classical Laplacian models standard diffusion mechanisms such as Brownian motion, the fractional Laplacian describes anomalous diffusion processes driven by long-range interactions and jump-type dynamics, as observed in Lévy flight processes \cite{Dip1,Val}. As a result, the operator \(\mathcal{L}\) embodies the interplay between local and nonlocal diffusion effects and provides a suitable framework for the mathematical modeling of phenomena in which short-range and long-range interactions occur simultaneously; we cite \cite{Dip1, Dip2, Dip3, Dip4} and references therein for more detail. Consequently, problem \eqref{eq1}
 serves as a framework for reaction-diffusion models with mixed diffusion.
 
Before delving into the subject of this paper, let us recall some known results in the classical polynomial case.
\subsection{The polynomial case}
The Cauchy problem \eqref{eq1} has been extensively studied  for $ \mathcal{L}=-\Delta$ in the framework of Lebesgue spaces, particularly when $f$ has polynomial growth, that is, $f(u)=|u|^{p-1}u, \,\,p>1.$ In this case, problem \eqref{eq1} becomes the classical semi-linear heat equation:
\begin{equation}\label{eq1classheat}
\left\{\begin{array}{ll}
\,\, \displaystyle {\partial_tu(x, t)-\Delta  u(x, t) =f(u(x, t)),} &\displaystyle {t>0,\,\,x\in {\mathbb{R}^d},}\\
{}\\
\displaystyle{u(0,x)=  u_0(x),\qquad\qquad}&\displaystyle{x\in {\mathbb{R}^d}.}
\end{array}
\right.
\end{equation} 
It is well established that if the initial datum \(u_0 \in L^\infty(\mathbb{R}^d)\) and the nonlinear term \(f \in C^1(\mathbb{R})\) satisfies \(f(0)=0\), then there exists a maximal time \(T=T(u_0)>0\) for which problem \eqref{eq1classheat} admits a unique local-in-time solution in $L^\infty\!\big(0,T; L^\infty(\mathbb{R}^d)\big)$. On the other hand, the first local well-posedness results for problem \eqref{eq1classheat} with the nonlinearity
$ f(u)=|u|^{p-1}u$, $p>1,$ and initial data $u_0$ which do not belong to $L^\infty(\mathbb{R}^d)$, were established by Weissler \cite{weissler1, weissler2}. He showed that there always exists a Lebesgue space \(L^{q}(\mathbb{R}^{d})\), with \(1 \leq q<\infty\), in which problem \eqref{eq1classheat} is locally well-posed. The scaling invariance property is a key feature in the analysis of power-type nonlinearities. A simple calculation shows that if $u(t, x)$ is a solution to problem  \eqref{eq1classheat}, then so is $u_\lambda(x, t)=\lambda^{2/(p-1)} u( \lambda x, \lambda^2 t), \, \lambda>0$. Moreover, the $L^{q}$-norm is invariant under this scaling if and only if $
q=q_{c}:=d(p-1)/2$, equivalently $p=p_h^*:= 1+2q/d$. The exponent  $p_h^*$ plays a crucial role in the existence theory for the Cauchy problem \eqref{eq1classheat}. In addition, Weissler showed that the Lebesgue space $L^{q_c}(\mathbb{R}^d)$ serves as the role of ``critical space'' for the well-posedness of problem \eqref{eq1classheat}.
With respect to this critical exponent, the existence and uniqueness theory for equation \eqref{eq1classheat} with nonlinearity
$ f(u)=|u|^{p-1}u$, $p>1,$ can be divided into the following two cases.

\medskip
\noindent \textbf{Case 1.} 
If $q \ge q_c$ with $q>1$, or $q>q_c$ with $q \ge 1$, then Weissler \cite{weissler2} and Brezis-Cazenave \cite{Brezis} proved that for any initial data $u_0 \in L^q(\mathbb{R}^d)$, there exists a positive time $T=T(u_0)$ and a unique solution
\[
u \in C([0,T],L^q(\mathbb{R}^d)) \cap L_{\mathrm{loc}}^{\infty}(0,T;L^{\infty}(\mathbb{R}^d))
\]
to problem (1.2).

\medskip
\noindent \textbf{Case 2.} 
If $q < q_c$, Weissler \cite{weissler2} and Brezis-Cazenave \cite{Brezis} showed that no local solution exists in any suitable weak sense. Furthermore, Haraux-Weissler \cite{HW} proved that uniqueness fails for the trivial initial data $u_0=0$ when $1+2/d < p < (d+2)/(d-2).$

 Now, we turn our attention towards the global existence of the solution of \eqref{eq1classheat}.   In his seminal work \cite{Fujita}, Fujita investigated problem \eqref{eq1} in the particular case of $f(u)=u^p$ and $\mathcal{L}=-\Delta$
to characterize the conditions under which global solutions exist and those under which they do not. He showed that for any nonnegative initial data \(u_0\), there are no global positive solutions if $1<p<p_F:=1+2/d.$ In contrast, when \(p>p_F\), the problem admits a global positive solution provided that the initial datum \(u_0\ge 0\) is sufficiently small. Later, Hayakawa \cite{Hayakawa} and Sugitani \cite{Sugitani} proved that the critical case \(p=p_F\) also leads to blow-up for all nonnegative initial data. The quantity $p_F$ is therefore referred to as the \emph{Fujita (critical) exponent}. In the critical case $q=q_{c}$ and $d \geq 3$, Weissler \cite{weissler3} proved a global existence result for problem \eqref{eq1classheat} under a smallness assumption on the initial data from $L^q(\mathbb{R}^d)$ with $1<q<\infty$.

More recently, several authors have considered problem \eqref{eq1} with \(f(u)=u^p\) and the fractional Laplacian
$
\mathcal{L}=(-\Delta)^s, \,\,\, s\in(0,1).
$
We refer to \cite{Fino, Guedda, Sugitani} and the references therein for a detailed account of these works. It has been observed that, in the sense of Fujita, the corresponding critical exponent in this fractional setting is given by
$p_c=1+2s/d.$

Finally, we turn our attention to the local and global well-posedness of the Cauchy problem \eqref{eq1} for polynomial type nonlinearities.   In this case, the global existence and blow-up of solutions have been investigated in recent works by Biagi, Punzo, and Vecchi \cite{Biagi}, as well as by Kumar and Torebek \cite{KT2026}. Similar results were obtained by Del Pezzo and Ferreira in \cite{Pezzo} for the operator $\mathcal L_{a,b} = -a\Delta + b(-\Delta)^s$, where $a,b \in \mathbb{R}_+$. 
In these studies, it was observed that the critical Fujita exponent is $p_c= 1 + 2s/d,$
which marks the threshold between global existence and finite-time blow-up of positive solutions to \eqref{eq1} for $u_0 \ge 0$. Interestingly, this critical exponent coincides with the Fujita exponent for the heat equation \eqref{eq1} associated with the pure fractional operator $\mathcal L = (-\Delta)^s$ with power-type nonlinearity. This shows that the presence of a local diffusion term does not modify the fundamental global existence behavior dictated by the fractional component. In particular, the mixed local-nonlocal operator retains the same criticality as the fractional Laplacian in determining the long-time dynamics of solutions.   An interesting phenomenon observed in \cite{KFA25} concerns the heat kernel estimates for mixed local–nonlocal operators (see Lemma \ref{themlplq}). Specifically, the local component dominates the behavior in the small-time regime, whereas the nonlocal component dominates at large times. This interplay yields sharper and novel results that do not arise when considering purely local or purely nonlocal operators separately.  We also cite \cite{BM25,Castillo3,KFA25,KT2025} and references therein for further developments in this line of research.

\subsection{The exponential case, main results and discussion} A natural question that arises in this context is the following: {\it when the nonlinearity exhibits exponential growth, can one identify functional spaces that contain unbounded initial data and still allow for the development of a well-posed existence theory?} This question has attracted considerable attention and has been investigated by several authors.

Cazenave and Weissler \cite{CW} pointed out a strong connection between the $L^q$-theory for the nonlinear heat equation and the $H^s$-theory for the nonlinear Schr\"odinger equation

\begin{equation}
    \begin{cases} \label{schro}
i\partial_t u + \Delta u = |u|^{p-1}u, & t>0,\; x \in \mathbb{R}^d, \\
u(0,x) = u_0(x), & x \in \mathbb{R}^d .
\end{cases}
\end{equation}
They showed that for $0 \le s < d/2$ there exists a critical exponent $p_s^{*}=1+4/(d-2s),$ which is determined by the scaling invariance of \eqref{schro} in the homogeneous Sobolev space $\dot H^s(\mathbb{R}^d)$. In particular, whenever $1<p\le p_s^{*}$ and $u_0\in H^s(\mathbb{R}^d)$, there exists a time $T=T(u_0)>0$ such that the Cauchy problem \eqref{schro} possesses a unique solution $u\in C([0,T);H^s(\mathbb{R}^d)).$ Furthermore, in the critical case $p=p_s^{*}$, global well-posedness can be obtained for sufficiently small initial data in $H^s(\mathbb{R}^d)$ by applying a fixed point argument. When $s=d/2$, the corresponding critical exponent becomes infinite, that is $p_s^{*}=\infty$, which indicates that the critical nonlinearity grows faster than any power-type nonlinearity. In this regime, Nakamura and Ozawa~\cite{NO98} established global existence for small initial data in $H^{d/2}(\mathbb{R}^d)$ for the nonlinear Schr\"odinger equation with an exponential-type nonlinearity satisfying $f(u)\sim e^{|u|^2}$.

Cazenave and Weissler \cite{CW} also observed that the critical exponent $p_h^*= 1+2q/d$, associated with the heat equation \eqref{eq1classheat} with initial data $u_0\in L^q(\mathbb{R}^d)$, and the critical exponent $p_s^*$ corresponding to the Schr\"odinger equation \eqref{schro}, are linked via the Sobolev embedding $\dot H^s(\mathbb{R}^d)\hookrightarrow L^q(\mathbb{R}^d),$
where $0\le s<d/2$ and the parameters satisfy
\[
\frac{1}{q}=\frac{1}{2}-\frac{s}{d}.
\]
A direct computation then shows that this embedding yields
\[
p_h^* = 1+\frac{2q}{d} = 1+\frac{4}{d-2s} = p_s^*,
\]
so that the functional spaces associated with the critical exponents of both equations coincide through this Sobolev embedding. Inspired by this relationship, Ruf and Terraneo~\cite{RT} suggested an analogue of this correspondence for problem \eqref{eq1} with $\mathcal{L}=-\Delta$ and $u_0\in L^q(\mathbb{R}^d)$, in the presence of an exponential nonlinearity satisfying
$f(u)\sim e^{|u|^2},\, |u|\to\infty$. They indicated that global solutions exist for sufficiently small initial data in a function space that contains $H^{d/2}(\mathbb{R}^d)$, by virtue of the Sobolev embedding. However, it is known that no Lebesgue space containing $H^{d/2}(\mathbb{R}^d)$ provides an appropriate framework for studying \eqref{eq1} with $\mathcal{L}=-\Delta$ and such exponential nonlinearities. Instead, the suitable functional setting for their analysis of the heat equation with exponential type nonlinearity was given by the Orlicz space $\exp L^2(\mathbb{R}^d)$. It is an appropriate place to introduce the Orlicz space $\exp L^p(\mathbb{R}^d), 1\le p<\infty$, as follows 
$$\exp L^p(\mathbb{R}^d)=\left\{u\in L^1_{\mathrm{loc}}(\mathbb{R}^d):\,\int_{\mathbb{R}^d}\left(\exp\left(\frac{|u(x)|^p}{\lambda^p}\right)-1\right)\,dx<\infty,\, \textnormal{for some}\,\, \lambda>0\right\},$$
which is a Banach space when endowed with the Luxemburg norm\label{Luxnorm}
$$\|u\|_{\exp L^p(\mathbb{R}^d)}=\inf\left\{\lambda>0:\, \int_{\mathbb{R}^d}\left(\exp\left(\frac{|u(x)|^p}{\lambda^p}\right)-1\right)\,dx\leq 1\right\}.$$
It follows from the Trudinger-Moser inequality that the Orlicz space 
$\exp L^{2}(\mathbb{R}^{d})$ enjoys the continuous embeddings
\[
H^{\frac{d}{2}}(\mathbb{R}^{d}) \hookrightarrow \exp L^{2}(\mathbb{R}^{d}) 
\hookrightarrow L^{r}(\mathbb{R}^{d}) \quad \text{for every}\,\, 2 \le r < \infty .
\]
Furthermore, in the scale of Orlicz spaces, $\exp L^{2}(\mathbb{R}^{d})$ can be regarded as 
the smallest space that contains $H^{d/2}(\mathbb{R}^{d})$. Ruf and Terraneo \cite{RT} investigated the heat equation \eqref{eq1} with a nonlinear term $f$  
exhibiting exponential growth $e^{|u|^2}$ at infinity. They obtained the local-in-time existence of 
solutions for sufficiently small initial data in the Orlicz space 
$\exp L^{2}(\mathbb{R}^{d})$. Later, Ioku \cite{Ioku2011} proved the global-in-time existence 
of solutions to the Cauchy problem \eqref{eq1} for small initial data belonging to 
$\exp L^{2}(\mathbb{R}^{d})$. It has further been established that when the nonlinearity satisfies 
$f(u) \sim e^{|u|^{s}}$ as $|u| \to \infty$ with $s>2$, the existence of 
solutions can no longer be guaranteed. In fact, in this regime one obtains 
a nonexistence result in the Orlicz space $\exp L^{2}(\mathbb{R}^{d})$; 
see, for instance, \cite{Ioku}. The two-dimensional case of the heat equation \eqref{eq1} was investigated by 
Ibrahim, Jrad, Majdoub, and Saanouni \cite{Ibrahim1}. This setting is critical 
from the viewpoint of the energy method as well as with respect to the 
critical Trudinger embedding. They established a local existence and uniqueness 
result for the corresponding problem. In addition, they obtained global 
existence in the defocusing case and derived a blow-up result in the focusing case. Majdoub and Tayachi \cite{Majdoub2,Majdoub3} extended the results of 
\cite{Ioku,Ioku2011} to nonlinearities of the type 
$f(u) \sim e^{|u|^{q}}$, as $|u|\to\infty$, with $q>1$. 
They studied the problem in the Orlicz framework 
$\exp L^{p}(\mathbb{R}^{d})$, where $p \ge q > 1$, 
establishing local well-posedness, global existence for small initial data, 
and large-time decay estimates of solutions. The second author and Kirane \cite{FK20} considered \eqref{eq1} for $\mathcal{L}=(-\Delta)^s$ extending the results of \cite{Majdoub2,Majdoub3} into the nonlocal setting. 

The main aim of this paper is to study problem \eqref{eq1} with a nonlinearity $f$ exhibiting exponential growth at infinity. As discussed earlier, to handle nonlinearities with exponential growth, one needs to work in the framework of Orlicz spaces, in analogy with the case of power-type nonlinearities, where Lebesgue spaces are the natural setting.  Let $e^{-t \mathcal{L}}$ denote the linear semigroup generated by the mixed local-nonlocal operator $\mathcal{L}$ defined in \eqref{heatsemimix}. Motivated by \cite{Ioku, Majdoub1,Majdoub2}, we study the local-wellposedness of the equation \eqref{eq1} in the space $\exp L^p_0(\mathbb{R}^d),$ which is defined by 
\begin{align*}\exp L^p_0(\mathbb{R}^d):=\Big\{u\in \exp L^p(\mathbb{R}^d): \textnormal{there exists $\{u_n\}_{n=1}^\infty\subset C^\infty_0(\mathbb{R}^d)$}\qquad\\
\quad\textnormal{such that $\lim_{n\rightarrow\infty}\|u_n-u\|_{\exp L^p(\mathbb{R}^d)}=0$}\Big\}.
\end{align*}
This space turns out to be an appropriate setting for the study of \eqref{eq1}, 
since $C_{0}^{\infty}(\mathbb{R}^{d})$ is not dense in 
$\exp L^{p}(\mathbb{R}^{d})$. This lack of density influences the mapping 
properties of the heat semigroup $e^{-t\mathcal{L}}$. In particular, 
$e^{-t\mathcal{L}}$ is bounded on $\exp L^{p}_{0}(\mathbb{R}^{d})$ 
(see Proposition~\ref{prop2}), whereas such a continuity property is not 
generally expected to hold on the whole space $\exp L^{p}(\mathbb{R}^{d})$ 
(cf.~\cite{Ioku} for the case $\mathcal{L}=-\Delta$). It is also known \cite{Ioku} that
$$\exp L^p_0(\mathbb{R}^d)=\left\{u\in L^1_{\rm{loc}}(\mathbb{R}^d):\,\int_{\mathbb{R}^d}\left(\exp\left(\alpha|u(x)|^p\right)-1\right)\,dx<\infty,\, \textnormal{for every}\,\, \alpha>0\right\}.$$

In order to state our main results,  we introduce two notions of mild solutions: 
The standard mild solution formulated in the space 
$\exp L^p_0(\mathbb{R}^d)$, 
and the weak-mild solution formulated in the larger space 
$\exp L^p(\mathbb{R}^d)$.

\begin{definition}\label{mild}
Let $u_0 \in \exp L^p_0(\mathbb{R}^d)$ and $T>0$. A function 
$u \in C([0,T], \exp L^p_0(\mathbb{R}^d))$ is called a {\it mild solution} of the Cauchy problem \emph{(\ref{eq1})} if it satisfies
\begin{equation}\label{eq2}
u(t)=e^{-t\mathcal{L}}u_0+\int_0^t e^{-(t-s)\mathcal{L}} f(u(s))\,ds,\qquad t\in[0,T].
\end{equation}
\end{definition}

\begin{definition}
Given $u_0 \in \exp L^p(\mathbb{R}^d)$ and $T>0$. A function $u \in L^\infty\big((0,T),\exp L^p(\mathbb{R}^d)\big)
$ is called a {\it weak-mild solution} of the Cauchy problem \emph{(\ref{eq1})} if satisfies the corresponding integral equation \eqref{eq2} in $\exp L^p(\mathbb{R}^d)$ for almost every $t\in(0,T)$, and
$u(t) \rightarrow u_0$ in the weak$^{*}$ topology  as $t \to 0$.
\end{definition}
Recall that $u(t)\to u_0$ in the weak$^{*}$ sense if and only if
\[
\lim_{t\to 0}\int_{\mathbb{R}^d}\big(u(t,x)-u_0(x)\big)\varphi(x)\,dx = 0,
\quad \text{for all } \varphi \in L^1(\ln L)^{1/p}(\mathbb{R}^d).
\]
Here, the space $L^1(\ln L)^{1/p}(\mathbb{R}^d)$ is defined by
\[
L^1(\ln L)^{1/p}(\mathbb{R}^d)
:= \left\{ f\in L^1_{\mathrm{loc}}(\mathbb{R}^d) \; : \;
\int_{\mathbb{R}^d} |f(x)|\,\ln^{1/p}\!\big(2+|f(x)|\big)\,dx < \infty \right\},
\]
which serves as a predual of $\exp L^p(\mathbb{R}^d)$; we refer to excellent books \cite{Bennett,Rao} for details and further references. 

We are now prepared to present the first main result of this paper, which addresses the local well-posedness of the Cauchy problem \eqref{eq1}. 
\begin{theorem}\label{theo1}\textup{(Local well-posedness)}. Let $d\ge 1$, $p>1$, and $s \in (0, 1)$. Suppose that the nonlinearity $f$ satisfies
 \begin{equation}\label{eq3}
  f(0)=0,\qquad |f(u)-f(v)|\leq C|u-v|(e^{\lambda |u|^p}+e^{\lambda |v|^p}),\quad \hbox{for all}\,\, u,v\in\mathbb{R}, 
  \end{equation} 
 for some constants, $C>0$ and $\lambda>0$. Then, for any initial data
$u_0 \in \exp L^p_0(\mathbb{R}^d)$, there exists a time $T=T(u_0)>0$ such that the Cauchy problem  \emph{(\ref{eq1})} admits a unique mild solution
$
u \in C\big([0,T],\exp L^p_0(\mathbb{R}^d)\big).
$
\end{theorem}
Condition \eqref{eq3} is fulfilled, for instance, by $f(u)=\pm u e^{|u|^p}$. In addition, by the uniqueness result of Theorem \ref{theo1}, the local solution $u$ of \eqref{eq1} 
can be extended to a maximal interval $[0,T_{\max})$ 
by a standard continuation argument 
(see, e.g., Cazenave and Haraux~\cite{CH}), where
\[
T_{\max}
:=
\sup\left\{
T>0  : \;
\text{there exists a mild solution }
u\in C([0,T],\exp L^p_0(\mathbb{R}^d))
\text{ to \eqref{eq1}}
\right\}
.
\]
Moreover, if $T_{\max}<\infty$, then the blow-up alternative holds:
\[
\lim_{t\to T_{\max}}
\|u(t)\|_{L^p(\mathbb{R}^d)\cap L^\infty(\mathbb{R}^d)}
=\infty.
\]

 The proof of Theorem \ref{theo1} relies on semigroup methods combined with a decomposition technique based on the density of $C_0^\infty(\mathbb{R}^d)$ in $\exp L^p_0(\mathbb{R}^d)$ (see, e.g.,~\cite{Ibrahim, Ioku, Majdoub1, Majdoub2}). This approach enables us to establish both existence and uniqueness of solutions to \eqref{eq1} in the $\exp L^p$ framework. Given $u_{0} \in \exp L^{p}_{0}(\mathbb{R}^{d})$, we decompose $u_{0} = v_{0} + w_{0}$, where $v_{0} \in C_{0}^{\infty}(\mathbb{R}^{d})$ and $w_{0} \in \exp L^{p}(\mathbb{R}^{d})$ is chosen to be sufficiently small in the corresponding norm. We first solve \eqref{eq1} with smooth initial data $v_{0}$ via classical semigroup arguments, obtaining a local-in-time bounded solution $v$. Setting $w := u - v$, we are led to a perturbed problem with small initial data. The latter is handled via the Banach fixed-point theorem applied to the integral formulation~\eqref{eq2} in a suitable Banach space adapted to $\exp L^p$. The argument relies crucially on semigroup estimates associated with the mixed local-nonlocal operator $\mathcal{L}$ (see \eqref{pqbound} and Proposition \ref{techprop1}), together with the growth condition \eqref{eq3}. This procedure yields the existence of the desired solution $u = v + w$. Uniqueness follows from similar estimates in the same functional framework. Further details and related techniques can be found in the literature (see, e.g., \cite{FK20,Ioku,Majdoub2}).
\begin{remark}
The solution obtained in Theorem~\ref{theo1} belongs to 
$L^\infty_{\mathrm{loc}}(0,T;L^\infty(\mathbb{R}^d))$. Indeed, let 
$
u\in C([0,T],\exp L^p_0(\mathbb{R}^d))
$
be a mild solution of \eqref{eq1}, that is, a solution of the integral equation \eqref{eq2}. 
Using the $L^p-L^\infty$ estimate \eqref{pqbound} together with 
Lemma~\ref{lemma1}-(3), we obtain
\[
\|e^{-t\mathcal{L}}u_0\|_{L^\infty}
\le C t_s^{-\frac{d}{2sp}}
\|u_0\|_{L^p}
\le C t_s^{-\frac{d}{2sp}}
\|u_0\|_{\exp L^p},
\qquad 0<t<T.
\]
Hence $e^{-t\mathcal{L}}u_0 \in  L^\infty_{\mathrm{loc}}(0,T;L^\infty(\mathbb{R}^d)$. It remains to estimate the nonlinear term. Fix $r>\max\{p,\frac{d}{2}\}$. 
By the $L^r-L^\infty$ estimate \eqref{pqbound}, we have
\[
\begin{aligned}
\left\|\int_0^t e^{-(t-\tau)\mathcal{L}} f(u(\tau))\, d\tau \right\|_{L^\infty}
&\le \int_0^t 
\|e^{-(t-\tau)\mathcal{L}} f(u(\tau))\|_{L^\infty}
\, d\tau \\
&\le C \int_0^t (t-\tau)_s^{-\frac{d}{2sr}}
\|f(u(\tau))\|_{L^r}
\, d\tau \\
&\le C\,t^{1-\frac{d}{2r}}
\sup_{0\le \tau\le T}
\|f(u(\tau))\|_{L^r}
<\infty,
\end{aligned}
\]
in view of Corollary~\ref{corollary1}. Consequently, $u \in L^\infty_{\mathrm{loc}}(0,T;L^\infty(\mathbb{R}^d))$.
Furthermore, if $f\in C^1(\mathbb{R})$, then the additional regularity of the solution implies
\[
u \in C([0,T],\exp L^p_0(\mathbb{R}^d))
\cap L^\infty_{\mathrm{loc}}(0,T;L^\infty(\mathbb{R}^d))
\]
satisfies \eqref{eq1} in the classical sense, namely,
$C^1$ in time and $C^2$ in space.
\end{remark}

We now examine the global existence and decay behavior of solutions to the Cauchy problem \eqref{eq1}. In this context, the behavior of the nonlinearity $f(u)$ in a neighborhood of the origin is of fundamental importance. To this end, we allow $f$ to satisfy the following growth condition near zero:
\[
|f(u)| \sim |u|^m,
\]
where the parameters are such that $d(m-1)/(2s) > 1$. The reason of this assumption is to reach the Fujita exponent $1+2s/d$ as obtained in \cite{KT2026}. More precisely, we assume that
\begin{equation}\label{eq4a}
f(0)=0,\qquad 
|f(u)-f(v)| \leq C\,|u-v|\left(|u|^{m-1}e^{\lambda |u|^p} + |v|^{m-1}e^{\lambda |v|^p}\right),
\end{equation}
for all $u,v \in \mathbb{R}$, where $C>0$, $\lambda>0$, and the exponents satisfy $d(m-1)/(2s) \geq p > 1$. A canonical example of a nonlinearity satisfying \eqref{eq4a} is $f(u)=\pm |u|^{m-1}u\,e^{|u|^p}$, with $m \geq 1+2sp/d$. We state our global existence result in the form of the following theorem.
\begin{theorem}\label{theo2}\textup{(Global existence).} Let $d \ge 1$ and $s\in(0,1)$. 
Assume that the nonlinearity $f$ satisfies \eqref{eq4a} for some $m \ge p$ and
$1 < p \le d(m-1)/(2s)$ with $C$ and $\lambda$ being positive constants.
Then there exists $\varepsilon > 0$ such that, for every initial data 
$u_0 \in \exp L^p(\mathbb{R}^d)$ with $\|u_0\|_{\exp L^p} \le \varepsilon,$ problem \eqref{eq1} admits a global weak-mild solution
$
u \in L^\infty\big((0,\infty), \exp L^p(\mathbb{R}^d)\big).
$
Moreover, the solution satisfies
\begin{equation}\label{eq5}
\lim_{t \to 0}
\|u(t) - e^{-t\mathcal{L}}u_0\|_{\exp L^p} = 0.
\end{equation}
In addition, there exists a constant $C>0$ such that
\begin{equation} \label{decay1.8}
    \|u(t)\|_{q} 
\le C\, t_s^{-(\frac{1}{m-1} - \frac{d}{2s q})},
\qquad \text{for all } t>0,
\end{equation}
where  $t_s:=\max\{t^s,t\}$ and the exponent $q$ satisfies:
\begin{itemize}
\item If $s < \dfrac{d(p-1)}{2p}$,\quad then\, $\dfrac{d(m-1)}{s} < 2q < \dfrac{d(m-1)}{s}\dfrac{1}{(2-m)_+}$.
\item If $s = \dfrac{d(p-1)}{2p}$,\quad then\, $\dfrac{d(m-1)}{s} < 2q < \dfrac{d(m-1)}{s}\dfrac{1}{(2-m)_+}$.
\item If $s > \dfrac{d(p-1)}{2p}$ and $(2-m)_+ < \dfrac{d(p-1)}{2ps}$,\quad then\, $$\dfrac{2(m-1)p}{p-1} < 2q < \dfrac{d(m-1)}{s}\dfrac{1}{(2-m)_+}.$$
\end{itemize}
Here, $(\cdot)_+$ denotes the positive part.
\end{theorem}
The proof of Theorem~\ref{theo2} is based on refined mapping estimates of the semigroup $\{e^{-t\mathcal{L}}\}_{t\geq 0}$ (see Proposition \ref{techprop1}), obtained via heat kernel bounds and interpolation arguments in Orlicz spaces. We then implement a fixed-point argument in a complete metric space defined through time-weighted $\exp L^{p}$ norms. This framework allows us to derive suitable a priori estimates and control the nonlinear term, leading to the global existence result for small initial data, as well as estimate \eqref{decay1.8}.

\begin{remark}
The decay estimate \eqref{decay1.8} highlights a unique asymptotic transition that bridges local and non-local diffusion theories. Its primary novelty lies in the short-time regime ($t \to 0$), where the solution mimics the decay behavior of the local heat equation through the effective power $\tilde{m}$ such that $\tilde{m}-1:= (m-1)/s,$ that is, $\tilde{m}:=(m-(1-s))/s$, aligning with the local smoothing effects discussed in \cite{Majdoub2,Majdoub3}. As time progresses ($t \to \infty$), the estimate naturally transitions to recover the purely non-local asymptotics of the fractional Laplacian $(-\Delta)^s$ established in \cite{FK20}. By unifying these two scales into a single estimate, \eqref{decay1.8} demonstrates that non-local effects are not instantaneous but emerge over time, providing a continuous interpolation between singular local dynamics and heavy-tailed fractional dissipation.
\end{remark}

\begin{remark}
For the analysis of Theorem~\ref {theo2}, we distinguish three cases according to the relationship between $s$ and the critical threshold $d(p-1)/(2 p)$, namely $s<d(p-1)/(2 p)$, $s>d(p-1)/(2 p)$, and $s=d(p-1)/(2 p)$.
This distinction reflects the balance between the smoothing effect of the semigroup $e^{-t\mathcal{L}}$ and the growth 
of exponential-type nonlinearity.  In particular, the threshold $~d(p-1)/(2p)$ arises naturally from the time-integrability condition required in the  estimate of the nonlinear term in the mild formulation.\\
 Moreover, in the case $s>d(p-1)/(2p)$, it is necessary to choose $m>p$. Indeed, taking $m=p$ would yield $s>d(m-1)/(2m)$, while condition $d(m-1)/s \ge 2p$ implies $s \le d(m-1)/(2m),$ leading to a contradiction. Therefore, the choice $m=p$ is not admissible in this regime.
\end{remark}

The paper is organized as follows.
In Section~\ref{sec2}, we recall the definition and basic properties of the Orlicz spaces $L^\phi(\mathbb{R}^d)$ and $L^\phi_{0}(\mathbb{R}^d)$, along with some fundamental properties of the heat semigroup $e^{-t\mathcal{L}}$. 
In Section~\ref{sec3}, we establish the mapping properties of the heat semigroup $e^{-t\mathcal{L}}$ in the Orlicz spaces $\exp L^p(\mathbb{R}^d)$ and $\exp L^p_0(\mathbb{R}^d)$. In Section~\ref{sec4}, we prove the local well-posedness result (Theorem~\ref{theo1}) for initial data $u_{0} \in \exp L^p_0(\mathbb{R}^d)$. Finally, in Section~\ref{sec5}, we establish the global existence result (Theorem~\ref{theo2}) for sufficiently small initial data $u_{0} \in \exp L^p(\mathbb{R}^d)$.

Throughout the paper, $C>0$ denotes a generic constant whose value may change from line to line. 
We write $X \lesssim Y$ (or equivalently $Y \gtrsim X$) to indicate that there exists a constant $C>0$ such that $X \leq C Y$. We write $X \approx Y$ if both $ X\lesssim Y$ and $Y\lesssim X$ hold. Unless otherwise specified, $L^q$ and $\exp L^q$ stands for   $L^q(\mathbb{R}^d)$  and $\exp L^q(\mathbb{R}^d)$, respectively, for $1\leq q\leq
\infty$. We denote by $\|\cdotp\|_q$ the usual $L^q(\mathbb{R}^d)$-norm, and by $\|\cdotp\|_{\exp L^q}$  the norm in the exponential Orlicz space $\exp L^q(\mathbb{R}^d)$.

\section{Preliminaries}\label{sec2}
This section is devoted to recalling the basic definitions and fundamental results concerning Orlicz spaces that will be needed in our analysis. 
For more comprehensive treatments and detailed presentations, we refer the reader to \cite{Adams,Rao,Trudinger}. We conclude this section by collecting some essential properties of  the heat semigroup associated with the mixed local-nonlocal operator $\mathcal{L}$.
\subsection{Orlicz space and its properties}
\begin{definition}[Orlicz space and subspaces]
Let $\phi:\mathbb{R}^+ \to \mathbb{R}^+$ be a convex and increasing function satisfying
\[
\phi(0)=0, 
\qquad 
\lim_{\tau\to0^+}\phi(\tau)=0, 
\qquad 
\lim_{\tau\to\infty}\phi(\tau)=\infty.
\]
Then the Orlicz space $L^\phi(\mathbb{R}^d)$ is defined by
\[
L^\phi(\mathbb{R}^d)
=
\left\{
u\in L^1_{\mathrm{loc}}(\mathbb{R}^d)
:\,
\int_{\mathbb{R}^d}
\phi\!\left(\frac{|u(x)|}{\lambda}\right)dx
<\infty,
\text{ for some } \lambda>0
\right\}.
\]
The space is endowed with the Luxemburg norm
\[
\|u\|_{L^\phi(\mathbb{R}^d)}
=
\inf\left\{
\lambda>0:
\int_{\mathbb{R}^d}
\phi\!\left(\frac{|u(x)|}{\lambda}\right)dx
\le 1
\right\}.
\]
We also introduce the subspace
\[
L^\phi_0(\mathbb{R}^d)
=
\left\{
u\in L^1_{\mathrm{loc}}(\mathbb{R}^d)
:\,
\int_{\mathbb{R}^d}
\phi\!\left(\frac{|u(x)|}{\lambda}\right)dx
<\infty,
\text{ for every } \lambda>0
\right\}.
\]
Clearly, $L^\phi_0(\mathbb{R}^d) \subset L^\phi(\mathbb{R}^d)$.
It is known (as in Ioku et al. \cite{Ioku}) that
$
L^\phi_0(\mathbb{R}^d)
=
\overline{C_0^\infty(\mathbb{R}^d)}^{\|\cdot\|_{L^\phi}},
$. Moreover, both 
$(L^\phi(\mathbb{R}^d),\|\cdot\|_{L^\phi})$ 
and 
$(L^\phi_0(\mathbb{R}^d),\|\cdot\|_{L^\phi})$ 
are Banach spaces.\\
\textbf{Special cases.}
\begin{itemize}
\item If $\phi(s)=s^p$ with $1\le p<\infty$, then 
$L^\phi(\mathbb{R}^d)=L^\phi_0(\mathbb{R}^d)=L^p(\mathbb{R}^d)$, 
and the Luxemburg norm coincides with the usual $L^p$ norm.
\item If $\phi(s)=e^{s^p}-1$ with $1\le p<\infty$, then 
$L^\phi(\mathbb{R}^d)$ coincides with the exponential Orlicz space 
$\exp L^p(\mathbb{R}^d)$, while 
$L^\phi_0(\mathbb{R}^d)$ corresponds to 
$\exp L^p_0(\mathbb{R}^d)$.
\end{itemize}
\noindent
\textbf{A basic property.}
Let $u\in L^\phi(\mathbb{R}^d)$ and set
$
K= \|u\|_{L^\phi} > 0.
$
By definition,
\[
K
=
\inf\left\{
\lambda>0:
\int_{\mathbb{R}^d}
\phi\!\left(\frac{|u(x)|}{\lambda}\right)dx
\le 1
\right\}.
\]
Since $\phi$ is increasing, if $\lambda$ satisfies
$
\displaystyle \int_{\mathbb{R}^d}
\phi\left(\frac{|u(x)|}{\lambda}\right)dx
\le 1,
$
then the same inequality holds for every $\lambda'>\lambda$.
Hence, the admissible set of $\lambda$ is of the form $[K,\infty)$, and in particular,
$$
\int_{\mathbb{R}^d}
\phi\!\left(
\frac{|u(x)|}{\|u\|_{L^\phi}}
\right)dx
\le 1.
$$
\end{definition}

The following lemma summarizes some embedding (inclusion) properties between Orlicz and Lebesgue spaces.
\begin{lem}[Inclusion properties]\label{lemma1} We have the following inclusion relations:

\begin{itemize}
\item[({1})] 
\cite[Lemma 2.3]{Majdoub2}
For every $1\leq q\leq p$, we have $$L^q(\mathbb{R}^d)\cap L^\infty(\mathbb{R}^d)\hookrightarrow \exp L^p_0(\mathbb{R}^d)\hookrightarrow \exp L^p(\mathbb{R}^d).$$ More precisely, we have 
\begin{equation}\label{eq6}
\|u\|_{\exp L^p}\leq\frac{1}{(\ln 2)^{1/p}}(\|u\|_q+\|u\|_\infty).
\end{equation}
\item[({2})] \cite[Lemma 2.3]{FK20}${}$
Let $\phi(\tau)=e^{\tau^p}-1-\tau^p$, $p>1$. For every $ q\leq 2p$, we have $L^q(\mathbb{R}^d)\cap L^\infty(\mathbb{R}^d)\hookrightarrow L^\phi_0(\mathbb{R}^d)\hookrightarrow L^\phi(\mathbb{R}^d)$, more precisely
$$
\|u\|_{ L^\phi}\leq C(p)(\|u\|_q+\|u\|_\infty).
$$
\item[({3})] \cite[Lemma 2.4]{Majdoub2}
For every $1\leq p\leq q<\infty$, we have $\exp L^p(\mathbb{R}^d)\hookrightarrow L^q(\mathbb{R}^d)$, more precisely
\begin{equation}\label{eq8}
\|u\|_q\leq\left(\Gamma\left(\frac{q}{p}+1\right)\right)^{1/q} \|u\|_{\exp L^p}.
\end{equation}
\end{itemize}
\end{lem}

\subsection{Heat semigroup of \texorpdfstring{$\mathcal{L}$}{L} and its properties}  For the material of this subsection, we cite  \cite{Biagi, Castillo3, KFA25} and references therein. We begin with the following result regarding the integral representation of the heat semigroup associated with the mixed local-nonlocal operator. 
\begin{lemma}[Heat kernel representation]
Let $e^{-t\mathcal{L}}$ be the heat semigroup on $L^{q}(\mathbb{R}^{d})$ $(q \geq 1)$ generated by $-\mathcal{L}$. Then
\begin{equation} \label{heatsemimix}
e^{-t\mathcal{L}} \phi(x)=\int_{\mathbb{R}^{d}} P^s_{t}(x-y)\,\phi(y)\,dy, \quad \phi \in L^{q}(\mathbb{R}^{d}),
\end{equation}
where
$$
P^s_{t}(x)=\frac{1}{(4\pi t)^{d/2}}\int_{\mathbb{R}^d} e^{-\frac{|x-\xi|^2}{4t}} H_t^{s}(\xi)\, d\xi, 
\quad t>0,\; x\in\mathbb{R}^d,
$$
is the heat kernel of $\mathcal{L}$. Here, $H_t^{s}$ denotes the heat kernel of the operator $-(-\Delta)^s$.
\end{lemma}
We have the following $L^q-L^r$ estimate of the heat semigroup associated with the mixed local-nonlocal operator.
\begin{lemma}[$L^q-L^r$ estimate] \label{themlplq}
There exists a constant $C>0$ such that for any $1 \leq q \leq r \leq \infty$ and any 
$\varphi \in L^q(\mathbb{R}^d)$,
\begin{equation} \label{pqbound}
\Vert e^{-t \mathcal{L}} \varphi \Vert_{r} \leq C\, t_s^{-\frac{d}{2s}(\frac{1}{q}-\frac{1}{r})}\Vert \varphi \Vert_{q},\quad t>0,
\end{equation}
where $t_s:=\max\{t^s,t\}$.
\end{lemma}
\section{Some technical results}\label{sec3}
This subsection aims to derive some important results that will be useful throughout the paper. Our first result concerns estimates for $e^{-t\mathcal{L}}$ on $\exp L^p(\mathbb{R}^d).$

\begin{proposition}\label{techprop1}${}$
Let $1\leq q\leq p$, $1\leq r\leq\infty$, and $0<s< 1$. Define $t_s:=\max\{t^s,t\}$. Then the heat semigroup $e^{-t \mathcal{L}}$ satisfies the following mapping properties:
\begin{enumerate}
\item $\left\|e^{-t\mathcal{L}} \varphi\right\|_{\exp L^p }\leq \left\|\varphi\right\|_{\exp L^p }$, \,\,\, for all $t>0,$\, $\varphi\in {\exp L^p(\mathbb{R}^d)}$.

\item $\left\|e^{-t\mathcal{L}} \varphi\right\|_{\exp L^p }\lesssim t_s^{-\frac{d}{2s q}}\left(\ln(t_s^{-\frac{d}{2s}}+1)\right)^{-1/p}\left\|\varphi\right\|_{q}$, \,\,\, for all $t>0,$\, $\varphi\in L^q(\mathbb{R}^d)$.

\item $\left\|e^{-t\mathcal{L}} \varphi\right\|_{\exp L^p}\lesssim \frac{1}{(\ln 2)^{1/p}}\left[t_s^{-\frac{d}{2s r}}\left\|\varphi\right\|_{r}+\left\|\varphi\right\|_{q}\right]$, \, for all $t>0,$\, $\varphi\in L^r(\mathbb{R}^d)\cap L^q(\mathbb{R}^d)$.
\end{enumerate}
\end{proposition}
\proof 
First, we establish (1). By the definition of the Luxemburg norm in the Orlicz space $\exp L^{p}(\mathbb{R}^{d})$ we have 
\begin{equation} \label{luxheatnorm}
    \left\|e^{-t\mathcal{L}} \varphi\right\|_{\exp L^p}=\inf\left\{\lambda>0: \,\, \int_{\mathbb{R}^d}\left(\exp\left(\frac{|e^{-t\mathcal{L}} \varphi|^p}{\lambda^p}\right)-1\right)\,dx\leq 1\right\}.
\end{equation}
Now, we first estimate the integral in \eqref{luxheatnorm} by using \eqref{pqbound} and the Taylor formula as follows:
\begin{align*}
\int_{\mathbb{R}^d}\left(\exp\left(\frac{|e^{-t\mathcal{L}} \varphi|^p}{\lambda^p}\right)-1\right)\,dx &=\int_{\mathbb{R}^d}\sum_{k=1}^{\infty}\left(\frac{|e^{-t\mathcal{L}} \varphi|^p}{\lambda^p}\right)^{k}\frac{1}{k!}\,dx=\sum_{k=1}^{\infty}\frac{\left\|e^{-t\mathcal{L}}\varphi\right\|_{pk}^{pk}}{k!\lambda^{pk}}\\&\leq \sum_{k=1}^{\infty}\frac{\left\|
\varphi\right\|_{pk}^{pk}}{k!\lambda^{pk}}
=\sum_{k=1}^{\infty}\frac{1}{k!\lambda^{pk}}\int_{\mathbb{R}^d}|\varphi|^{pk}dx\\
&=\int_{\mathbb{R}^d}\left(\exp\left(\frac{| \varphi|^p}{\lambda^p}\right)-1\right)\,dx,
\end{align*}
for any $\lambda>0$. Therefore, it follows that
\begin{align*}
&\left\{\lambda>0:\,\,\int_{\mathbb{R}^d}\left(\exp\left(\frac{|\varphi|^p}{\lambda^p}\right)-1\right)\,dx\leq 1\right\} \\& \quad\quad\subseteq \left\{\lambda>0:\,\,\int_{\mathbb{R}^d}\left(\exp\left(\frac{|e^{-t\mathcal{L}} \varphi|^p}{\lambda^p}\right)-1\right)\,dx\leq 1\right\},
\end{align*}
which further implies, using \eqref{luxheatnorm}, that
$$
\left\|e^{-t\mathcal{L}} \varphi\right\|_{\exp L^p}\leq \inf\left\{\lambda>0:\,\,\int_{\mathbb{R}^d}\left(\exp\left(\frac{|\varphi|^p}{\lambda^p}\right)-1\right)\,dx\leq 1\right\}=\left\|\varphi\right\|_{\exp L^p}.
$$
Next, we prove (2).  We first notice, using \eqref{pqbound}  and the Taylor expansion, that
\begin{eqnarray*}
\int_{\mathbb{R}^d}\left(\exp\left(\frac{|e^{-t\mathcal{L}} \varphi|^p}{\lambda^p}\right)-1\right)\,dx&=&\sum_{k=1}^{\infty}\frac{\left\|e^{-t\mathcal{L}} \varphi\right\|_{pk}^{pk}}{k!\lambda^{pk}}\\
&\leq& \sum_{k=1}^{\infty}\frac{C^{pk}t_s^{-\frac{d}{2s}(\frac{1}{q}-\frac{1}{pk})pk}\left\|\varphi\right\|_{q}^{pk}}{k!\lambda^{pk}}\\
&=&t_s^{\frac{d}{2s}}\left(\exp\left(\frac{C t_s^{-\frac{d}{2s q}}\|\varphi\|_q}{\lambda}\right)^p-1\right),
\end{eqnarray*}
for any $\lambda>0$. Moreover, it is easy to see that
$$t_s^{\frac{d}{2s}}\left(\exp\left(\frac{C t_s^{-\frac{d}{2 s q}}\|\varphi\|_q}{\lambda}\right)^p-1\right)\leq 1\Longleftrightarrow \lambda\geq C\,t_s^{-\frac{d}{2s q}}\left(\ln( t_s^{-\frac{d}{2 s}}+1)\right)^{-1/p}\left\|\varphi\right\|_{q},$$
which shows that
\begin{eqnarray*}
&{}&\left\{\lambda>0 :\,\, \lambda\in \Big[ C\,t_s^{-\frac{d}{2 s q}}\left(\ln( t_s^{-\frac{d}{2 s}}+1)\right)^{-1/p}\left\|\varphi\right\|_{q}, \,\infty \Big)\right\}\\
&{}&\subseteq \left\{\lambda>0 :\,\, \int_{\mathbb{R}^d}\left(\exp\left(\frac{|e^{-t\mathcal{L}} \varphi|^p}{\lambda^p}\right)-1\right)\,dx\leq 1\right\}.
\end{eqnarray*}
Therefore, using the Luxembourg norm \eqref{luxheatnorm} and the above inclusion, we have
\begin{eqnarray*}
\left\|e^{-t\mathcal{L}} \varphi\right\|_{\exp L^p}&\leq&\inf \left\{\lambda>0 :\,\, \lambda\in \Big[ C\,t_s^{-\frac{d}{2s q}}\left(\ln( t_s^{-\frac{d}{2s}}+1)\right)^{-1/p}\left\|\varphi\right\|_{q}, \,\infty \Big)\right\}\\
&=&C\,t_s^{-\frac{d}{2s q}}\left(\ln( t_s^{-\frac{d}{2s}}+1)\right)^{-1/p}\left\|\varphi\right\|_{q}.
\end{eqnarray*}
Finally, we derive (3). Applying Lemma \ref{lemma1}-(1) and then using \eqref{pqbound}, we conclude that
\begin{eqnarray*}
\left\|e^{-t \mathcal{L}} \varphi\right\|_{\exp L^p}&\leq&\frac{1}{(\ln 2)^{1/p}}\left(\left\|e^{-t\mathcal{L}} \varphi\right\|_q+\left\|e^{-t\mathcal{L}} \varphi\right\|_\infty\right)\\
&\leq&\frac{1}{(\ln 2)^{1/p}}\left(\left\| \varphi\right\|_q+C\,t_s^{-\frac{d}{2s r}}\left\| \varphi\right\|_r\right).
\end{eqnarray*}
This completes the proof. \hfill$\square$\\

We also require the following smoothing result. Although its proof follows along the same lines as that of \cite[Proposition~3.7]{Majdoub1}, we provide a detailed argument here for the sake of completeness.

\begin{proposition}\label{prop2}
For any $\varphi\in \exp L^p_0(\mathbb{R}^d)$, the map $t\mapsto e^{-t\mathcal{L}}\varphi$ belongs to $C([0,\infty),\exp L^p_0(\mathbb{R}^d)).$
\end{proposition}
\proof 
   Let $\varphi\in \exp L^p_0(\mathbb{R}^d)$. By Proposition \ref{techprop1}-(1) and the definition of $\exp L^p_0(\mathbb{R}^d)$, it follows that $e^{-t\mathcal{L}}\varphi\in \exp L^p_0(\mathbb{R}^d)$ for all $t>0$. By the linearity of the semigroup $e^{-t\mathcal{L}}$, it is sufficient to prove the continuity at $t=0$, that is,
$$
\lim_{t\rightarrow 0}\left\|e^{-t\mathcal{L}} \varphi-\varphi\right\|_{\exp L^p}=0.
$$ Since $\varphi\in \exp L^p_0(\mathbb{R}^d)$, there exists a sequence $(\varphi_n)_n\subset C^\infty_0(\mathbb{R}^d)$ such that
$$
\lim_{n\rightarrow \infty}\left\|\varphi_n-\varphi\right\|_{\exp L^p}=0.
$$
By virtue of (\ref{eq6}) and Proposition \ref{techprop1}-(1), we infer that
\begin{eqnarray*}
&{}&\left\|e^{-t\mathcal{L}} \varphi-\varphi\right\|_{\exp L^p}\\
&{}&\leq\left\|e^{-t\mathcal{L}} (\varphi-\varphi_n)\right\|_{\exp L^p}+\left\|e^{-t\mathcal{L}} \varphi_n-\varphi_n\right\|_{\exp L^p}+\left\|\varphi_n-\varphi\right\|_{\exp L^p}\\
&{}&\leq\frac{1}{(\ln 2)^{1/p}}\left(\|e^{-t\mathcal{L}} \varphi_n-\varphi_n\|_p+\|e^{-t\mathcal{L}} \varphi_n-\varphi_n\|_\infty\right)
+2\left\|\varphi_n-\varphi\right\|_{\exp L^p}.
\end{eqnarray*}
Since $\varphi_n\in C^\infty_0(\mathbb{R}^d)$ and the semigroup $e^{-t\mathcal{L}}$ is  strongly continuous on $L^r(\mathbb{R}^d)$ for $1<r\leq\infty$, we obtain
$$
\lim_{t\rightarrow 0}\left(\|e^{-t\mathcal{L}} \varphi_n-\varphi_n\|_p+\|e^{-t\mathcal{L}} \varphi_n-\varphi_n\|_\infty\right)=0.
$$
Therefore,
$$
\limsup_{t\rightarrow 0}\left\|e^{-t\mathcal{L}} \varphi-\varphi\right\|_{\exp L^p}
\leq 2\left\|\varphi_n-\varphi\right\|_{\exp L^p},
$$
for every $n\in\mathbb{N}$. This completes the proof of the proposition. 
\hfill$\square$\\

Since $e^{-t\mathcal{L}}$ is a $C^0$-semigroup on $L^p(\mathbb{R}^d),$ using Proposition \ref{prop2} we obtain the following result.
\begin{corollary}The semigroup $\{e^{-t\mathcal{L}}\}_{t\geq0}$ is a $C^0$-semigroup on $\exp L^p_0(\mathbb{R}^d).$
    \end{corollary}

We state the following result (see \cite[Lemma 4.1.5]{CH}), adapted to the mixed local-nonlocal heat semigroup. 
\begin{lemma}\label{lemma10}
Let $X$ be a Banach space and $ g\in L^1(0,T;X)$, then $$\displaystyle \int_0^te^{-(t-\tau)\mathcal{L}}g(\tau)\,d\tau\in C([0,T],X).$$ Moreover, we have
$$\left\|\int_0^te^{-(t-\tau)\mathcal{L}}g(\tau)\,d\tau\right\|_{L^\infty(0,T;X)}\leq \|g\|_{L^1(0,T;X)}.$$
\end{lemma}
The following lemmas are essential for the proof of the global existence (Theorem \ref{theo2}).
\begin{lemma}${}$\label{lemma6}\cite[Lemma~2.6]{Majdoub2}
Let $\lambda>0$, $1\leq q<\infty$ and $K>0$ be such that $\lambda q K^p\leq 1$. Assume that $u\in\exp L^p(\mathbb{R}^d)$ satisfies
$\|u\|_{\exp L^p}\leq K,$
then $$\exp\left(\frac{|u|^p}{\lambda^p}\right)-1\in L^q(\mathbb{R}^d)\quad \text{and}\quad
\left\| e^{\lambda|u|^p}-1\right\|_{q}\leq \left(\lambda q K^p\right)^{1/q}.$$
\end{lemma}
\begin{lemma}${}$\label{lemma4}
Let $p>1$, $s 
\in (0, 1)$ be such that $s<d(p-1)/(2 p)$. Then, for every $r>\max\{d/2,1\}$, there exists $C=C(d,p,\tau,r)$ such that
 $$\left\|\int_0^te^{-(t-\tau)\mathcal{L}} g(\tau)\,d\tau\right\|_{L^\infty(0,\infty;\, \exp L^p)}\leq C\|g\|_{L^\infty(0,\infty;\,L^1\cap L^r)},$$ 
for every $g\in L^\infty(0,\infty;\, L^1(\mathbb{R}^d)\cap L^r(\mathbb{R}^d))$.
\end{lemma}
\proof By Proposition \ref{techprop1}, we have
\begin{equation}\label{eq10}
\left\|e^{-t\mathcal{L}} \varphi\right\|_{\exp L^p}\leq C\,t_s^{-\frac{d}{2 s}}\left(\ln( t_s^{-\frac{d}{2 s}}+1)\right)^{-1/p}\left\|\varphi\right\|_{1},
\end{equation}
and
\begin{equation}\label{eq11}
\left\|e^{-t\mathcal{L}} \varphi\right\|_{\exp L^p}\leq C\,(t_s^{-\frac{d}{2 s r}}+1)\left[\left\|\varphi\right\|_{r}+\left\|\varphi\right\|_{1}\right].
\end{equation}
for all $t>0$, $\varphi\in L^1(\mathbb{R}^d)\cap L^r(\mathbb{R}^d)$, where  $t_s:=\max\{t^s,t\}$. Combining \eqref{eq10} and \eqref{eq11}, we get
$$\left\|e^{-t\mathcal{L}} \varphi\right\|_{\exp L^p}\lesssim \kappa(t)\left[\left\|\varphi\right\|_{r}+\left\|\varphi\right\|_{1}\right],$$
where 
$$\kappa(t):=\min\left\{(t_s^{-\frac{d}{2 s r}}+1),t_s^{-\frac{d}{2 s }}\left(\ln( t_s^{-\frac{d}{2 s}}+1)\right)^{-1/p}\right\}.$$
Due to the assumptions $s<d(p-1)/(2 p)$ and $r>d/2$, we see that $\kappa\in L^1(0,\infty)$. Thus, for $g\in L^\infty(0,\infty;L^1(\mathbb{R}^d)\cap L^r(\mathbb{R}^d))$, we have
\begin{eqnarray*}
\left\|\int_0^te^{-(t-\tau)\mathcal{L}} g(\tau)\,d\tau\right\|_{\exp L^p}&\leq& \int_0^t\left\|e^{-(t-\tau)\mathcal{L}} g(\tau)\right\|_{\exp L^p}\,d\tau\\
&\leq& \int_0^t\kappa(t-\tau)\left(\left\|g(\tau)\right\|_{1}+\left\|g(\tau)\right\|_{r}\right)\,d\tau\\
&\leq&\|g\|_{L^\infty(0,\infty;L^1\cap L^r)} \int_0^\infty\kappa(\tau)\,d\tau,
\end{eqnarray*}
for every $t>0$. This proves Lemma \ref{lemma4}. \hfill$\square$\\

We note that $d(p-1)/(2 p)$ may not be included in $(0,1)$. So, if $d(p-1)/(2 p)>1$, we have $d(p-1)/(2 p)>s$, and this case is recovered by Lemma \ref{lemma4}. If $d(p-1)/(2 p)\leq1$, we have three cases to study: the case of $s<d(p-1)/(2 p)$ is done according to Lemma \ref{lemma4}, and the case $s >d(p-1)/(2 p)$ can be done separately without using any kind of an a priori estimate, so it remains to study the case of $s=d(p-1)/(2 p)$ where we have a similar result as in Lemma \ref{lemma4}.  For this, we need to introduce an appropriate Orlicz space. Let $L^\phi(\mathbb{R}^d)$ be this space, with $\phi(u)=e^{|u|^p}-1-|u|^p$, associated with its Luxemburg norm. From the definition of $\|\cdotp\|_{L^\phi}$, (\ref{eq8}), and the standard inequality $e^{\theta \tau}-1\leq\theta(e^\tau-1)$, $0\leq\theta\leq1$, $\tau\geq 0$, we can easily get
\begin{equation}\label{eq0}
 \|u\|_{L^p}+\|u\|_{L^\phi}  \approx  \|u\|_{\exp L^p}.
\end{equation}
Indeed, we have the following result.
\begin{lemma}${}$\label{lemma5}
Let $p>1$ and $s\in(0,1)$ be such that $s=d(p-1)/(2 p)$. Then there exists a constant $C=C(d,p)$ such that
 $$\left\|\int_0^te^{-(t-\tau)\mathcal{L}} g(\tau)\,d\tau\right\|_{L^\infty(0,\infty;L^\phi)}\leq C\|g\|_{L^\infty(0,\infty;L^1\cap L^{2p}\cap L^{2p'})},$$ 
for every $g\in L^\infty(0,\infty;L^1(\mathbb{R}^d)\cap L^{2p}(\mathbb{R}^d)\cap L^{2p'}(\mathbb{R}^d))$, where $p'=p/(p-1)$.
\end{lemma}
\proof On the one hand, by \ref{pqbound}, we have
\begin{eqnarray*}
\int_{\mathbb{R}^d}\phi\left(\frac{|e^{-t\mathcal{L}} \varphi|}{\lambda}\right)\,dx&=&\sum_{k=2}^{\infty}\frac{\left\|e^{-t\mathcal
L} \varphi\right\|_{pk}^{pk}}{k!\lambda^{pk}}\\
&\leq& \sum_{k=2}^{\infty}\frac{C^{pk}t_s^{-\frac{d}{2 s}(1-\frac{1}{pk})pk}\left\|\varphi\right\|_{1}^{pk}}{k!\lambda^{pk}}\\&=&t_s^{\frac{d}{2 s}}\phi\left(\frac{C t_s^{-\frac{d}{2 s}}\|\varphi\|_1}{\lambda}\right)\\
&\leq& t_s^{\frac{d}{2 s}}\left(\exp\left(\frac{C t_s^{-\frac{d}{2 s}}\|\varphi\|_1}{\lambda}\right)^{2p}-1\right),
\end{eqnarray*}
for all $\lambda>0$, $t>0$, $\varphi\in L^1(\mathbb{R}^d)$, where we have used the fact that $e^{|x|^p}-1-|x|^p\leq e^{|x|^{2p}}-1$, for all $x\in\mathbb{R}$. As
$$ t_s^{\frac{d}{2 s}}\left(\exp\left(\frac{C t_s^{-\frac{d}{2 s}}\|\varphi\|_1}{\lambda}\right)^{2p}-1\right)\leq 1\Longleftrightarrow \lambda\geq C\,t_s^{-\frac{d}{2 s}}\left(\ln( t_s^{-\frac{d}{2 s}}+1)\right)^{-1/2p}\left\|\varphi\right\|_{1};$$
hence, 
\begin{eqnarray*}
\left\{\lambda>0:\, \lambda\in [ C\,t_s^{-\frac{d}{2 s}}\left(\ln( t_s^{-\frac{d}{2 s}}+1)\right)^{-1/2p}\left\|\varphi\right\|_{1};\infty)\right\}
\,\subseteq \left\{\lambda>0:\, \int_{\mathbb{R}^d}\phi\left(\frac{|e^{-t\mathcal{L}} \varphi|}{\lambda}\right)\,dx\leq 1\right\};
\end{eqnarray*}
whereupon
\begin{equation}\label{eq12}
\left\|e^{-t\mathcal{L}} \varphi\right\|_{L^\phi}\leq C\,t_s^{-\frac{d}{2 s}}\left(\ln( t_s^{-\frac{d}{2 s}}+1)\right)^{-1/2p}\left\|\varphi\right\|_{1},
\end{equation}
for all $t>0,$\, $\varphi\in L^1(\mathbb{R}^d)$. On the other hand, from \ref{pqbound} and Lemma \ref{lemma1}-(2), we have
\begin{eqnarray}\label{eq13}
\left\|e^{-t\mathcal{L}} \varphi\right\|_{L^\phi}&\leq&\left\|e^{-t\mathcal{L}} \varphi\right\|_{\infty}+\left\|e^{-t\mathcal{L}} \varphi\right\|_{2p}\nonumber\\
&\leq& C t^{-\frac{d(p-1)}{4 sp}}\left\|\varphi\right\|_{2p'}+\left\|\varphi\right\|_{2p}\nonumber\\
&\lesssim& (t_s^{-\frac{1}{2}}+1)\left(\left\|\varphi\right\|_{2p'}+\left\|\varphi\right\|_{2p}\right),
\end{eqnarray}
for all $t>0$, $\varphi\in L^{2p}(\mathbb{R}^d)\cap L^{2p'}(\mathbb{R}^d)$, where we have used the fact that $s=d(p-1)/(2 p)$. Now, letting $g\in L^\infty(0,\infty;L^1(\mathbb{R}^d)\cap L^{2p}(\mathbb{R}^d)\cap L^{2p'}(\mathbb{R}^d))$, we conclude from (\ref{eq12}) and (\ref{eq13}) that
$$
\left\|e^{-t\mathcal{L}} g(t)\right\|_{L^\phi}\lesssim \kappa(t) \left\|g(t)\right\|_{L^1\cap L^{2p}\cap L^{2p'}},
$$
for all $t>0$, where
$$\kappa(t) :=\min\left\{(t_s^{-\frac{1}{2}}+1),t_s^{-\frac{d}{2 s}}\left(\ln( t_s^{-\frac{d}{2 s}}+1)\right)^{-1/2p}\right\}.$$
We can easily check that $\kappa\in L^1(0,\infty)$. Therefore,
\begin{eqnarray*}
\left\|\int_0^te^{-(t-\tau)\mathcal{L}} g(\tau)\,d\tau\right\|_{L^\phi}&\leq& \int_0^t\left\|e^{-(t-\tau)\mathcal{L}} g(\tau)\right\|_{L^\phi}\,d\tau\\
&\leq& \int_0^t\kappa(t-\tau)\left\|g(\tau)\right\|_{L^1\cap L^{2p}\cap L^{2p'}}\,d\tau\\
&\leq&\|g\|_{L^\infty(0,\infty;L^1\cap L^{2p}\cap L^{2p'})} \int_0^\infty\kappa(\tau)\,d\tau,
\end{eqnarray*}
for every $t>0$. This proves Lemma \ref{lemma5}. \hfill$\square$\\

 
 \section{Local well-posedness: Proof of Theorem \ref{theo1}}\label{sec4}
 In this section, we prove Theorem \ref{theo1} concerning the local well-posedness of problem \eqref{eq1}. We begin by stating several auxiliary lemmas that are essential for the proof in the space $\exp L^p_0(\mathbb{R}^d)$.

 \begin{lemma}\cite[Proposition~2.9]{Majdoub2}.\label{prop3} 
Let $1\leq p<\infty$ and $u\in C([0,T],\exp L^p_0(\mathbb{R}^d))$ for some $T>0$. Then, for every $\alpha>0$, it holds
$$\left(e^{\alpha|u|^p}-1\right)\in C([0,T],L^r(\mathbb{R}^d)),\quad 1\leq r<\infty.$$
\end{lemma}
\begin{lemma}\cite[Corollary~2.13]{Majdoub2}.\label{corollary1}
Let $1\leq p<\infty$ and $u\in C([0,T],\exp L^p_0(\mathbb{R}^d))$ for some $T>0$. Assume that $f$ satisfies $(\ref{eq3})$. Then, for every $p\leq r<\infty$, it holds
$$f(u)\in C([0,T],L^r(\mathbb{R}^d)).$$
\end{lemma}
\begin{lemma}${}$\label{lemma7} 
Let $s\in(0,1)$, $p>1$ and $v_0\in L^p(\mathbb{R}^d)\cap L^\infty(\mathbb{R}^d)$. Suppose that the nonlinearity $f$ satisfies condition \emph{(\ref{eq3})}. Then, there exist a time $T=T(v_0)>0$ and a mild solution $v\in C([0,T],\exp L^p_0(\mathbb{R}^d))\cap L^\infty(0,T;L^\infty(\mathbb{R}^d))$ of 
\begin{equation}\label{eq24}
\left\{\begin{array}{ll}
\,\, \displaystyle {\partial_t v+\mathcal{L}v =f(v),} &\displaystyle {t>0,x\in {\mathbb{R}^d},}\\
{}\\
\displaystyle{v(0)=  v_0 \qquad\qquad}&\displaystyle{x\in {\mathbb{R}^d}}.
\end{array}
\right.
\end{equation} 
 \end{lemma}
 \begin{proof} We prove this lemma by applying the Banach fixed-point theorem.  For this purpose, we first introduce the following complete metric space
$$Y_T:=\left\{v\in C([0,T],\exp L^p_0(\mathbb{R}^d))\cap L^\infty(0,T;L^\infty(\mathbb{R}^d)):\,\,\|v\|_{Y_T}\leq 2\|v_0\|_{L^p\cap L^\infty}\right\},$$
where $\|v\|_{Y_T}:=\|v\|_{L^\infty(0,T;L^p)}+\|v\|_{L^\infty(0,T;L^\infty)}$ and $\|v_0\|_{L^p\cap L^\infty}:=\|v_0\|_{L^p}+\|v_0\|_{L^\infty}$. For $v\in Y_T$, we define the operator
$$\Phi(v):=e^{-t\mathcal{L}}v_0+\int_0^te^{-(t-s)\mathcal{L}}f(v(\tau))\,d\tau.$$
We will show that, for $T>0$ sufficiently small, the mapping $\Phi$ is a contraction on $Y_T$.\\
First, we verify that $\Phi:Y_T\rightarrow Y_T$. Let $v\in Y_T$. Since $v_0\in L^p(\mathbb{R}^d)\cap L^\infty(\mathbb{R}^d)$,  Lemma \ref{lemma1}-(1) implies that $v_0\in\exp L^p_0(\mathbb{R}^d)$. Consequently, by Proposition \ref{prop2}, we have $e^{-t\mathcal{L}}v_0\in C([0,T],\exp L^p_0(\mathbb{R}^d))$. Next, for $q=p$ or $q=\infty$, we have
\begin{equation}\label{eq26}
\|f(v)\|_{q}\leq Ce^{\lambda\|v\|^p_\infty}\|v\|_{q}\leq Ce^{\lambda\|v\|^p_\infty}(2\|v_0\|_{L^p\cap L^\infty}),
\end{equation}
which, together with Lemma \ref{lemma1}-(1), yields $f(v)\in \exp L^p_0(\mathbb{R}^d)$. More precisely, we obtain $f(v)\in L^1(0,T;\exp L^p_0(\mathbb{R}^d))$ . Therefore, by density arguments and the smoothing effect of the fractional semigroup $e^{-t\mathcal{L}}$ (Lemma \ref{lemma10}), it follows that 
$$\int_0^te^{-(t-s)\mathcal{L}}f(v(s))\,ds\in C([0,T],\exp L^p_0(\mathbb{R}^d)).$$
Thus, $\Phi(v)\in C([0,T],\exp L^p_0(\mathbb{R}^d))$. Moreover, using (\ref{pqbound}) and (\ref{eq26}), we obtain
$$\|\Phi(v)\|_{Y_T}\leq \|v_0\|_{L^p\cap L^\infty} +2TC(2\|v_0\|_{L^p\cap L^\infty})e^{\lambda(2\|v_0\|_{L^p\cap L^\infty})^p}\leq 2\|v_0\|_{L^p\cap L^\infty},$$
provided that $T>0$ is sufficiently small, namely $4TCe^{\lambda(2\|v_0\|_{L^p\cap L^\infty})^p}\leq 1$. This shows that $\Phi(v)\in Y_T$.\\
We now prove that $\Phi$ is a contraction. Let $v_1,v_2\in Y_T$. For $q=p$ or $q=\infty$, we have
$$
\|f(v_1)-f(v_2)\|_{q}\leq C\|v_1-v_2\|_q(e^{\lambda\|v_1\|^p_\infty}+e^{\lambda\|v_2\|^p_\infty})\leq 2C\|v_1-v_2\|_{Y_T} e^{\lambda(2\|v_0\|_{L^p\cap L^\infty})^p}.
$$
Using (\ref{pqbound}), we deduce
$$
\|\Phi(v_1)-\Phi(v_2)\|_{Y_T}\leq 2TC\|v_1-v_2\|_{Y_T} e^{\lambda(2\|v_0\|_{L^p\cap L^\infty})^p}\leq \frac{1}{2}\|v_1-v_2\|_{Y_T},
$$
provided that $T>0$ is sufficiently small. Hence, $\Phi$ is a contraction on $Y_T$. Therefore, by the Banach fixed point theorem, $\Phi$ admits a unique fixed point $u \in Y_T$. In particular, $u$ is a mild solution of the considered problem. This completes the proof of Lemma~\ref{lemma7}.
\end{proof}
\begin{lemma}\cite[Lemma 4.4]{Majdoub2}.\label{lemma9}
Let $v\in L^\infty(0,T;L^\infty(\mathbb{R}^d))$ for some $T>0$. Suppose that the nonlinearity $f$ satisfies condition \emph{(\ref{eq3})} with some $\lambda>0.$ Let $1<p\leq q<\infty$, and $w_1,w_2\in \exp L^p(\mathbb{R}^d)$ with $\|w_1\|_{\exp L^p},\|w_2\|_{\exp L^p}\leq M$ for sufficiently small $M>0$ (namely $2^p\lambda q M^p\leq1$). Then, there exists a constant $C=C(q)>0$ such that 
$$\|f(w_1+v)-f(w_2+v)\|_q\leq Ce^{2^{p-1}\lambda\|v\|_\infty^p}\|w_1-w_2\|_{\exp L^p}.$$
 \end{lemma}

\begin{lemma}${}$\label{lemma8}
Let $s\in(0,1)$, $p>1$, and $w_0\in \exp L^p_0(\mathbb{R}^d)$. Suppose that the nonlinearity $f$ satisfies condition \emph{(\ref{eq3})}. Let $T>0$ and $v\in L^\infty(0,T;L^\infty(\mathbb{R}^d))$ be given by Lemma $\ref{lemma7}$. Then, for $\|w_0\|_{\exp L^p}\leq\varepsilon$, with $\varepsilon\ll1$ small enough, there exist a time $\widetilde{T}=\widetilde{T}(w_0,\varepsilon, v)>0$ and a mild solution $w\in C([0,\widetilde{T}],\exp L^p_0(\mathbb{R}^d))$ to the problem 
\begin{equation}\label{eq25}
\left\{\begin{array}{ll}
\,\, \displaystyle {\partial_t w+\mathcal{L}w =f(w+v)-f(v),} &\displaystyle {t>0,x\in {\mathbb{R}^d},}\\
{}\\
w(0)= w_0, & \displaystyle{ x\in {\mathbb{R}^d}.}
\end{array}
\right.
\end{equation} 
 \end{lemma}
\begin{proof} This lemma is also proved by applying the Banach fixed-point theorem. For $\widetilde{T}>0$, we define
$$W_{\widetilde{T}}:=\left\{w\in C([0,\widetilde{T}],\exp L^p_0(\mathbb{R}^d)):\,\,\|w\|_{L^\infty(0,\widetilde{T};\exp L^p_0)}\leq 2\varepsilon\right\},$$
and we introduce the mapping $\Phi$ on $ W_{\widetilde{T}}$ by
$$\Phi(w):=e^{-t\mathcal{L}}w_0+\int_0^te^{-(t-\tau)\mathcal{L}}\left(f(w(\tau)+v(\tau))-f(v(\tau))\right)\,d\tau.$$
We first show that $\Phi$ is a contraction on $W_{\widetilde{T}}$ for sufficiently small $\varepsilon$ and $\widetilde{T}>0$. Let $w_1,w_2\in W_{\widetilde{T}}$. Using Lemma \ref{lemma1}-(1), we obtain
\begin{equation}\label{eq27}
\|\Phi(w_1)-\Phi(w_2)\|_{\exp L^p}\leq \frac{1}{(\ln 2)^{1/p}}\left(\|\Phi(w_1)-\Phi(w_2)\|_{p}+\|\Phi(w_1)-\Phi(w_2)\|_{\infty}\right).
\end{equation}
Let $r>0$ be a parameter chosen such that $r>\max\{p,\frac{d}{2}\}$. Then, by the $L^r-L^\infty$ estimate (\ref{pqbound}), we have
$$\|\Phi(w_1)-\Phi(w_2)\|_{\infty}\leq C\int_0^t(t-\tau)_s^{-\frac{d}{2 s r}}\|f(w_1(\tau)+v(\tau))-f(w_2(\tau)+v(\tau))\|_r\,d\tau.$$
Applying Lemma \ref{lemma9} and assuming that $2^p\lambda r (2\varepsilon)^p\leq1$, together with the estimate $(t-\tau)_s\geq (t-s)^s$, we obtain
\begin{eqnarray}\label{eq28}
\|\Phi(w_1)-\Phi(w_2)\|_{\infty}&\leq& Ce^{2^{p-1}\lambda\|v\|_\infty^p}\left(\int_0^t(t-\tau)^{-\frac{d}{2 r}}\,d\tau\right)\|w_1-w_2\|_{L^\infty(0,\widetilde{T};\exp L^p)}\nonumber\\
&\leq& Ce^{2^{p-1}\lambda\|v\|_\infty^p}\widetilde{T}^{1-\frac{d}{2 r}}\|w_1-w_2\|_{L^\infty(0,\widetilde{T};\exp L^p)}.
\end{eqnarray}
In a similar way, under the condition $2^p\lambda p (2\varepsilon)^p\leq1$, we get
\begin{eqnarray}\label{eq29}
\|\Phi(w_1)-\Phi(w_2)\|_{p}&\leq& \int_0^t\left\|f(w_1(\tau)+v(\tau))-f(w_2(\tau)+v(\tau))\right\|_{p}\,d\tau\nonumber\\
&\leq& Ce^{2^{p-1}\lambda\|v\|_\infty^p}\widetilde{T}\|w_1-w_2\|_{L^\infty(0,\widetilde{T};\exp L^p)}.
\end{eqnarray}
Combine estimates (\ref{eq28}) and (\ref{eq29}) with (\ref{eq27}), we deduce that, for $\varepsilon\ll1$ sufficiently small,
\begin{equation}\label{eq30}
\|\Phi(w_1)-\Phi(w_2)\|_{\exp L^p}\leq \frac{1}{2}\|w_1-w_2\|_{L^\infty(0,\widetilde{T};\exp L^p)},
\end{equation}
provided that $\widetilde{T}\ll1$ is chosen sufficiently small so that $Ce^{2^{p-1}\lambda\|v\|_\infty^p}\left(\widetilde{T}+\widetilde{T}^{1-\frac{d}{2 r}}\right)\leq 1/2$.\\
We now show that $\Phi$ maps $W_{\widetilde{T}}$ into itself. Let $w\in W_{\widetilde{T}}$. Since $w_0\in L^p(\mathbb{R}^d)\cap L^\infty(\mathbb{R}^d)$, Lemma \ref{lemma1}-(1) ensures that $w_0\in\exp L^p_0(\mathbb{R}^d)$. Consequently, by Proposition \ref{prop2}, 
$$e^{-t\mathcal{L}}w_0\in C([0,T],\exp L^p_0(\mathbb{R}^d)).$$ 
Next, applying estimates (\ref{eq28})-(\ref{eq29}) with $w_1=w$ and $w_2=0$, and assuming that $2^p\lambda r (2\varepsilon)^p\leq1$, we deduce that
$$\Phi(w)-e^{-t\mathcal{L}}w_0\in L^\infty(0,T;\exp L^p_0(\mathbb{R}^d)),$$
thanks to the embedding $L^p(\mathbb{R}^d)\cap L^\infty(\mathbb{R}^d)\hookrightarrow \exp L^p_0(\mathbb{R}^d)$ (Lemma \ref{lemma1}-(1)). By the smoothing effect of the fractional semigroup $e^{-t\mathcal{L}}$ (Lemma \ref{lemma10}), it follows that 
$$\Phi(w)-e^{-t\mathcal{L}}w_0\in C([0,T],\exp L^p_0(\mathbb{R}^d)),$$
and hence $\Phi(w)\in C([0,T],\exp L^p_0(\mathbb{R}^d))$.  Moreover, using Proposition \ref{techprop1} and estimate (\ref{eq30}) with $w_1=w$ and $w_2=0$, we obtain for $T>0$ sufficiently small that 
$$\|\Phi(w)\|_{W_{\widetilde{T}}}\leq \|w_0\|_{\exp L^p} +\frac{1}{2}\|w\|_{L^\infty(0,\widetilde{T};\exp L^p)}\leq 2\varepsilon.$$
This establishes that $\Phi(w)\in W_{\widetilde{T}}$. The conclusion then follows from the Banach fixed-point theorem. This completes the proof of Lemma \ref{lemma8}.
\end{proof}

{\bf Proof of Theorem~\ref{theo1}.} To prove the existence of the desired solution, let $u_0 \in \exp L^p_0(\mathbb{R}^d)$. Then, for every $\varepsilon > 0$, there exists $v_0 \in C^\infty_0(\mathbb{R}^d)$ such that
$\|w_0\|_{\exp L^p}\leq\varepsilon,$ where $w_0:=u_0-v_0$. We now split problem (\ref{eq1}) into two subproblems, \eqref{eq24} and \eqref{eq25}, where the first corresponds to smooth initial data and the second to small initial data in $\exp L^p(\mathbb{R}^d)$. Let $r>\max\left\{p,d/2\right\}$ and fix $\varepsilon>0$ such that $2^p\lambda r (2\varepsilon)^p\leq1$. By Lemma \ref{lemma7}, there exist $0<T_1=T_1(\|v_0\|_{L^p\cap L^\infty})\ll1$
and a mild solution $v\in C([0,T_1],\exp L^p_0(\mathbb{R}^d))
\cap L^\infty(0,T_1;L^\infty(\mathbb{R}^d))$ to problem \eqref{eq24}, satisfying $\|v\|_{L^\infty(0,T_1;L^p\cap L^\infty)}
\le 2\|v_0\|_{L^p\cap L^\infty}.$ Choose $\widetilde{T}\in(0,T_1)$ such that
\[
C e^{2^{\,2p-1}\lambda\|v_0\|_{L^p\cap L^\infty}^p}
\left(\widetilde{T}
+\widetilde{T}^{1-\frac{d}{2r}}\right)
\le \frac12.
\]
Then, by Lemma \ref{lemma8}, there exists a mild solution $w\in C([0,\widetilde{T}];\exp L^p_0(\mathbb{R}^d))$ to  \eqref{eq25}.
Setting $u=v+w$, we obtain $ u\in C([0,\widetilde{T}];\exp L^p_0(\mathbb{R}^d))$ which is a mild solution of \eqref{eq1}.\\
To prove uniqueness, let $u,v \in C([0,T],\exp L^p_0(\mathbb{R}^d))$ 
be two mild solutions of \eqref{eq1} on $[0,T]$ with the same initial data 
$u(0)=v(0)=u_0$. Define
\[
t_0=\sup\left\{t\in[0,T]:\, u(\tau)=v(\tau)\ \text{for all } \tau\in[0,t]\right\}.
\]
Assume that $0\le t_0<T$. By continuity, it follows that 
$u(t_0)=v(t_0)$. Define the shifted functions
$
\tilde{u}(t)=u(t+t_0), \,\,
\tilde{v}(t)=v(t+t_0).
$
Then $\tilde{u},\tilde{v}\in C([0,T-t_0],\exp L^p_0(\mathbb{R}^d))$, 
and they satisfy \eqref{eq2} on $(0,T-t_0]$, with $\tilde{u}(0)=\tilde{v}(0)=u(t_0).$
We will show that there exists $\tilde{t}\in(0,T-t_0]$ such that
\begin{equation}\label{eq31}
\sup_{0<t<\tilde{t}}
\|\tilde{u}(t)-\tilde{v}(t)\|_{\exp L^p}
\le C(\tilde{t})
\sup_{0<t<\tilde{t}}
\|\tilde{u}(t)-\tilde{v}(t)\|_{\exp L^p},
\end{equation}
with a constant $C(\tilde{t})<1$. This immediately implies 
$\tilde{u}=\tilde{v}$ on $[0,\tilde{t}]$, and hence 
$u(t+t_0)=v(t+t_0)$ for all $t\in[0,\tilde{t}]$, 
which contradicts the definition of $t_0$.\\
To establish \eqref{eq31}, we estimate separately the $L^p$ and $L^\infty$ norms 
of $\tilde{u}-\tilde{v}$. By \eqref{pqbound} and \eqref{eq3}, we have
\[
\|\tilde{u}(t)-\tilde{v}(t)\|_{p}\le \int_0^t
\|f(\tilde{u}(\tau))-f(\tilde{v}(\tau))\|_{p}\, d\tau\le C \int_0^t 
\|(\tilde{u}(\tau)-\tilde{v}(\tau))
(e^{\lambda |\tilde{u}(\tau)|^p}
+e^{\lambda |\tilde{v}(\tau)|^p})\|_{p}
\, d\tau.
\]
Splitting the exponential terms and using H\"older's inequality, we obtain
\[
\begin{aligned}
\|\tilde{u}(t)-\tilde{v}(t)\|_{p}
&\le 2 \int_0^t 
\|\tilde{u}(\tau)-\tilde{v}(\tau)\|_{p}
\, d\tau  \\
&\quad + \int_0^t 
\|\tilde{u}(\tau)-\tilde{v}(\tau)\|_{q}
\left\|(e^{\lambda |\tilde{u}(\tau)|^p}-1)
+(e^{\lambda |\tilde{v}(\tau)|^p}-1)\right\|_{r}
\, d\tau,
\end{aligned}
\]
where $1/q+1/r=1/p$.
Since $q\ge p$, Lemma~\ref{lemma1}-(3) yields
\[
\begin{aligned}
\|\tilde{u}(t)-\tilde{v}(t)\|_{p}
&\le C t 
\sup_{0<\tau<t}
\|\tilde{u}(\tau)-\tilde{v}(\tau)\|_{\exp L^p}  \\
&\quad + C \sup_{0<\tau<t}
\|\tilde{u}(\tau)-\tilde{v}(\tau)\|_{\exp L^p}
\int_0^t
\left\|(e^{\lambda |\tilde{u}(\tau)|^p}-1)
+(e^{\lambda |\tilde{v}(\tau)|^p}-1)\right\|_{r}
\, d\tau.
\end{aligned}
\]
Moreover, by Proposition~\ref{prop3},
$$
\sup_{0<\tau<T-t_0}
\left\|(e^{\lambda |\tilde{u}(\tau)|^p}-1)
+(e^{\lambda |\tilde{v}(\tau)|^p}-1)\right\|_{r}
\le C(T,t_0,\tilde{u},\tilde{v})<\infty.
$$
Combining the above estimates, we obtain
\begin{equation}\label{eq32}
\sup_{0<\tau<t}
\|\tilde{u}(\tau)-\tilde{v}(\tau)\|_{p}
\le C(T,t_0,\tilde{u},\tilde{v})\, t\,
\sup_{0<\tau<t}
\|\tilde{u}(\tau)-\tilde{v}(\tau)\|_{\exp L^p}.
\end{equation}
Using the $L^r-L^\infty$ estimate (\ref{pqbound}), we similarly obtain, for $r>\max\{p,d/2\}$,
$$
\|\tilde{u}(t)-\tilde{v}(t)\|_\infty \leq C\int_0^t(t-\tau)_s^{-\frac{d}{2s r}}\|f(\tilde{u}(\tau))-f(\tilde{v}(\tau))\|_r\,d\tau.
$$
Proceeding as above and using Lemma \ref{lemma1}-(3)  and Proposition \ref{prop3}, we get
\begin{equation}\label{eq34}
\sup_{0<\tau<t}\|\tilde{u}(\tau)-\tilde{v}(\tau)\|_{ \infty}\leq C(T,t_0,\tilde{u},\tilde{v})t^{1-\frac{d}{2 r}} \sup_{0<\tau<t}\|\tilde{u}(\tau)-\tilde{v}(\tau)\|_{\exp L^p}.
\end{equation}
Finally, combining \eqref{eq32} and \eqref{eq34} and using Lemma~\ref{lemma1}-(1), we deduce
\[
\sup_{0<\tau<t}
\|\tilde{u}(\tau)-\tilde{v}(\tau)\|_{\exp L^p}
\le
C(T,t_0,\tilde{u},\tilde{v})
\left(t+t^{1-\frac{d}{2r}}\right)
\sup_{0<\tau<t}
\|\tilde{u}(\tau)-\tilde{v}(\tau)\|_{\exp L^p}.
\]
By choosing $t>0$ sufficiently small, we obtain \eqref{eq31}. This completes the proof.
\hfill$\square$

\section{Global Well-posedness: Proof of Theorem \ref{theo2}}\label{sec5}
In this section, we prove the global well-posedness and decay estimates for problem \eqref{eq1}  with initial data in the space $\exp L^p(\mathbb R^d)$.

{\bf Proof of Theorem \ref{theo2}.}  We prove the theorem by distinguishing two cases: the first case corresponds to $s < d(p-1)/(2p),$ while the second corresponds to $s \geq d(p-1)/(2p).$\\
\textbf{(i) Case $s < d(p-1)/(2p)$.} We prove the result by means of the Banach fixed-point theorem. To this end, we introduce the complete metric space
\[
\mathcal{U}_\varepsilon
=
\left\{
u \in L^\infty(0,\infty;\exp L^p(\mathbb{R}^d))
\;:\;
\|u\|_{L^\infty(0,\infty;\exp L^p(\mathbb{R}^d))}
\le 2\varepsilon
\right\},
\]
and define the integral operator $\Phi$ by
\[
\Phi(u)
:=
e^{-t\mathcal{L}}u_0
+
\int_0^t e^{-(t-\tau)\mathcal{L}} f(u(\tau))\, d\tau.
\]
Here $\varepsilon>0$ is a sufficiently small constant, 
which will be chosen later so that
$
\|u_0\|_{\exp L^p}
\le \varepsilon.
$
To establish the result, we show that $\Phi$ is a contraction 
mapping on $\mathcal{U}_\varepsilon$ for a suitable choice of 
$\varepsilon$. We begin by proving that $\Phi$ is well-defined on $\mathcal{U}_\varepsilon$, that is, $\Phi:\mathcal{U}_\varepsilon\rightarrow \mathcal{U}_\varepsilon$. For $u\in \mathcal{U}_\varepsilon$, using Proposition \ref{techprop1}-(1) and Lemma \ref{lemma4}, we have for every $r>\max\{d/2,1\}$,
$$
\left\|\Phi(u)\right\|_{\exp L^p}\leq\varepsilon+C\left\|f(u)\right\|_{L^\infty(0,\infty;L^1\cap L^r)}.
$$
We now estimate \( f(u) \) in the \( L^q(\mathbb{R}^d) \)-norm for cases \( q = 1 \) and \( q = r \).
From assumption \eqref{eq4a}, we have
\[
|f(u)| \le  C |u|^{m}\big(e^{\lambda |u|^p}-1\big) + C |u|^{m}.
\]
Applying H\"older's inequality, we obtain
$$
\|f(u)\|_{q}
\leq C \|u\|_{2mq}^{m}
\|e^{\lambda |u|^p}-1\|_{2q}
+ C \|u\|_{mq}^{m}.
$$
Using Lemma~\ref{lemma1}-(3) and the fact that $m \ge p$, we deduce
\[
\|f(u)\|_{q}
\le
C \|u\|_{\exp L^p}^{m}
\|e^{\lambda |u|^p}-1\|_{2q}
+ C \|u\|_{\exp L^p}^{m}.
\]
In addition, by Lemma~\ref{lemma6} and the fact that $u \in \mathcal{U}_\varepsilon$, we obtain
\begin{equation}\label{eq19}
\|f(u)\|_{q}
\le
C (2\varepsilon)^m
\left( 1 + (2\lambda q (2\varepsilon)^p)^{1/q} \right)
\le
C (2\varepsilon)^m
\left( 1 + (2\lambda q (2\varepsilon)^p)^{1/r} \right).
\end{equation}
Consequently, choosing $\varepsilon > 0$ sufficiently small, we arrive at
$\|\Phi(u)\|_{\exp L^p}
\le 2\varepsilon$, that is, $\Phi(u) \in \mathcal{U}_\varepsilon$.\\
We now establish the contraction property of $\Phi$ on $\mathcal{U}_\varepsilon$. Let $u,v \in \mathcal{U}_\varepsilon$. For every  $r>\max\{d/2,1\}$, Lemma~\ref{lemma4} yields
$$
\left\|\Phi(u)-\Phi(v)\right\|_{\exp L^p}\leq
C\left\|f(u)-f(v)\right\|_{L^\infty(0,\infty;
L^1\cap L^r)}.
$$
To estimate $f(u)-f(v)$ in the $L^q$-norm for $q=1$ and $q=r$, we proceed as follows. Using assumption~(\ref{eq4a}), we obtain
\begin{eqnarray*}
|f(u)-f(v)|
&\leq&
C|u-v|
\left(|u|^{m-1}(e^{\lambda |u|^p}-1)
+|v|^{m-1}(e^{\lambda |v|^p}-1)\right) \\
&\quad&
+\,C|u-v|\left(|u|^{m-1}+|v|^{m-1}\right).
\end{eqnarray*}
Applying H\"older's inequality with the conjugate exponent  $m'=m/(m-1)$, we deduce
\begin{eqnarray*}
C\left\|f(u)-f(v)\right\|_{q}
&\le \underbrace{C \left\|u-v\right\|_{mq}
\left\|
|u|^{m-1}(e^{\lambda |u|^p}-1)
+|v|^{m-1}(e^{\lambda |v|^p}-1)
\right\|_{q m'}}_{I}\\
&\underbrace{+C\left\|u-v\right\|_{mq}
\left\||u|^{m-1}+|v|^{m-1}\right\|_{q m'}}_{J}.
\end{eqnarray*}
Using H\"older's inequality once more, together with 
Lemma~\ref{lemma1}-(3) and the assumption $m\ge p$, we estimate $I$ as follows
$$
I \le
C \left\|u-v\right\|_{\exp L^p} 
\left(
\|u\|^{m-1}_{\exp L^p}
\left\|e^{\lambda |u|^p}-1\right\|_{2 q m'}
+
\|v\|^{m-1}_{\exp L^p}
\left\|e^{\lambda |v|^p}-1\right\|_{2 q m'}
\right).
$$
Invoking Lemma~\ref{lemma6} and the fact that 
$u,v\in \mathcal{U}_\varepsilon$, we infer that
\[
I
\le
C2^m\varepsilon^{m-1}
\left(
\frac{2\lambda mq}{m-1}(2\varepsilon)^p
\right)^{1/2qm'}
\left\|u-v\right\|_{\mathcal{U}_\varepsilon}
\le
\frac18
\left\|u-v\right\|_{\mathcal{U}_\varepsilon},
\]
for $\varepsilon>0$ sufficiently small.
Similarly, we estimate $J$ as
$$
J
\le C2^m\varepsilon^{m-1}
\left\|u-v\right\|_{\exp L^p}
\le
\frac18
\left\|u-v\right\|_{\mathcal{U}_\varepsilon},
$$
again for $\varepsilon>0$ sufficiently small. Consequently,
\[
C\left\|f(u)-f(v)\right\|_{L^1 \cap L^r}
\le
2(I+J),
\]
and therefore
\[
\left\|\Phi(u)-\Phi(v)\right\|_{\exp L^p}
\le
\frac12
\left\|u-v\right\|_{\mathcal{U}_\varepsilon}.
\]
This concludes the proof of the global existence result stated in 
Theorem~\ref{theo2} for the case  $s<d(p-1)/(2p)$. To derive the decay estimate \eqref{decay1.8}, we proceed along the same lines as in the contraction argument for the case $s \ge d(p-1)/(2p)$ below, with a slight modification of the functional framework. Instead of working in $\mathcal{U}_\varepsilon$, we consider the complete metric space
$$\left\{u\in L^\infty(0,\infty;\exp L^p(\mathbb{R}^d));\,\,\sup_{t>0}t_s^\sigma\|u(t)\|_{L^q(\mathbb{R}^d)}+\|u\|_{L^\infty(0,\infty;\exp L^p(\mathbb{R}^d))}\leq M\varepsilon\right\},$$
equipped with the metric $d(u,v):=\sup_{t>0}t_s^\sigma\|u(t)-v(t)\|_{q}$, where $M>0$ is a sufficiently large constant and $\varepsilon>0$ is chosen small enough so that $\|u_0\|_{\exp L^p}\leq\varepsilon$. The parameters $\sigma$ and $q$ are selected to satisfy 
$$ \sigma=\frac{1}{m-1}-\frac{d}{2s q}>0\quad\hbox{and}\quad\frac{d(m-1)}{s}<2q<\frac{d(m-1)}{s}\frac{1}{(2-m)_+}.$$
\noindent\textbf{(ii) Case $s \geq d(p-1)/(2p)$.} 
\label{subsec4.2} We now turn to the proof of the global existence result stated in Theorem~\ref{theo2} for the regime $s \ge d(p-1)/(2p)$. The argument follows the same approach as in \cite{FK20,Majdoub2}, based on a contraction mapping principle in a suitable complete metric space. We introduce the set
\begin{eqnarray*}
\mathcal{U}_\varepsilon =
\left\{
u \in L^\infty(0,\infty;\exp L^p(\mathbb{R}^d))
\;:\;
\sup_{t>0} t_s^\sigma \|u(t)\|_{L^q(\mathbb{R}^d)}
+
\|u\|_{L^\infty(0,\infty;\exp L^p(\mathbb{R}^d))}
\le M\varepsilon
\right\},
\end{eqnarray*}
where $t_s:=\max\{t^s,t\}$, and $\sigma=1/(m-1) - d/(2s q)>0$. Here, $M>0$ is a sufficiently large constant and $0<\varepsilon\ll1$ is chosen small enough so that
$
\|u_0\|_{\exp L^p}
\le \varepsilon.
$
By \cite[Proposition~2.2]{Majdoub2}, 
it follows that $(\mathcal{U}_\varepsilon, d)$ is a complete metric space endowed with the distance $
d(u,v)
=
\sup_{t>0}
t^\sigma
\|u(t)-v(t)\|_{q}.
$
For any $u \in \mathcal{U}_\varepsilon$, we define the  operator $\Phi$ by
\[
\Phi(u)
=
e^{-t\mathcal{L}}u_0
+
\int_0^t e^{-(t-\tau)\mathcal{L}} f(u(\tau))\, d\tau.
\]
The aim  is to show that $\Phi$ is a contraction map on $\mathcal{U}_\varepsilon$. We begin by showing that $\Phi$ maps $\mathcal{U}_\varepsilon$ into itself. Let $u\in \mathcal{U}_\varepsilon$. By Proposition \ref{techprop1}, we have
$$\|e^{-t\mathcal{L}}u_0\|_{\exp L^p}\leq \|u_0\|_{\exp L^p}\leq\varepsilon.$$
Next, we estimate the $L^q$-norm. Choosing
 $$\sigma=\frac{d}{2 s}\left(\frac{2 s}{d(m-1)}-\frac{1}{q}\right)=\frac{1}{m-1}-\frac{d}{2 s q}>0,$$
for $q>d(m-1)/(2s)\geq p$, and applying \eqref{pqbound} together with Lemma \ref{lemma1}-(3), we obtain
$$
t_s^\sigma\|e^{-t\mathcal{L}}u_0\|_{L^q}\leq C\,\|u_0\|_{L^{\frac{d(m-1)}{2 s}}}\leq C\,\|u_0\|_{\exp L^p}\leq C\varepsilon.
$$
To estimate the second term of $\Phi(u)$ in the space $\exp L^p(\mathbb{R}^d)$, we first consider the case $s=d(p-1)/(2 p)$. Recall from (\ref{eq0}) that
$\|u\|_{L^p}+\|u\|_{L^\phi}  \approx \|u\|_{\exp L^p},$
where $\phi(u)=e^{|u|^p}-1-|u|^p$. Hence, it suffices to establish the following two estimates:
\begin{equation}\label{eq113}
\left\|\int_0^te^{-(t-\tau)\mathcal{L}} f(u(\tau))\,d\tau\right\|_{L^\infty(0,\infty;L^p(\mathbb{R}^d))}=O(\varepsilon),
\end{equation}
and
\begin{equation}\label{eq14}
\left\|\int_0^te^{-(t-\tau)\mathcal{L}}f(u(\tau))\,d\tau\right\|_{L^\infty(0,\infty;L^\phi(\mathbb{R}^d))}=O(\varepsilon).
\end{equation}
We now proceed with the proof of (\ref{eq113}). Since
$$|f(u)|\leq C|u|^me^{\lambda|u|^p}=C|u|\sum_{k=0}^{\infty}\frac{\lambda^k}{k!}|u|^{kp+m-1},$$
using (\ref{pqbound}) and H\"older's inequality, with $1\leq r\leq p$, and $\frac{1}{r}=\frac{1}{p}+\frac{1}{a}$,  it follows that
\begin{eqnarray*}
\left\|\int_0^te^{-(t-\tau)\mathcal{L}} f(u(\tau))\,d\tau\right\|_{p}&\leq& C\int_0^t(t-\tau)_s^{-\frac{d}{2 sa}}\|f(u(\tau))\|_{r}\,d\tau\\
&\leq& C\sum_{k=0}^{\infty}\frac{\lambda^k}{k!}\int_0^t(t-\tau)_s^{-\frac{d}{2 sa}}\|u(\tau)\|_{p}\|u(\tau)\|^{kp+m-1}_{{(kp+m-1)a}}\,d\tau,
\end{eqnarray*}
Next, applying H\"older's interpolation inequality together with Lemma \ref{lemma1}-(3), we obtain
\begin{eqnarray*}
&{}&\left\|\int_0^te^{-(t-\tau)\mathcal{L}} f(u(\tau))\,d\tau\right\|_{p}\\
&{}&\leq C\sum_{k=0}^{\infty}\frac{\lambda^k}{k!}\int_0^t(t-\tau)_s^{-\frac{d}{2 sa}}\|u(\tau)\|_{p}\|u(\tau)\|^{(kp+m-1)\theta}_{q}\|u(\tau)\|^{(kp+m-1)(1-\theta)}_{\rho}\,d\tau\\
&{}& \leq C\sum_{k=0}^{\infty}\frac{\lambda^k}{k!}\left(\Gamma\left(\frac{\rho}{p}+1\right)\right)^{\frac{(kp+m-1)(1-\theta)}{\rho}}\\
&{}&\quad\times \int_0^t(t-\tau)_s^{-\frac{d}{2 sa}}\tau_s^{-\sigma(kp+m-1)\theta}\left(\tau_s^{\sigma}\|u(\tau)\|\right)^{(kp+m-1)\theta}_{q} \|u(\tau)\|^{(kp+m-1)(1-\theta)+1}_{\exp L^p}\,d\tau ,\qquad
\end{eqnarray*}
where
\begin{equation}\label{propA}
\frac{1}{a(kp+m-1)}=\frac{\theta}{q}+\frac{1-\theta}{\rho},\quad 0\leq \theta\leq 1,\quad\hbox{and}\quad p\leq\rho<\infty.
\end{equation}
Therefore, since $u\in \mathcal{U}_\varepsilon$, we obtain the following estimate
\begin{eqnarray}\label{eq15}
&{}&\left\|\int_0^te^{-(t-\tau)\mathcal{L}} f(u(\tau))\,d\tau\right\|_{p}\nonumber\\
&{}& \leq C\sum_{k=0}^{\infty}\frac{\lambda^k}{k!}\left(\Gamma\left(\frac{\rho}{p}+1\right)\right)^{\frac{(kp+m-1)(1-\theta)}{\rho}} (M\varepsilon)^{kp+m}\int_0^t(t-\tau)^{-\frac{d}{2 sa}}\tau^{-\sigma(kp+m-1)\theta}\,d\tau\nonumber\\
&{}&=C\sum_{k=0}^{\infty}\frac{\lambda^k}{k!}\left(\Gamma\left(\frac{\rho}{p}+1\right)\right)^{\frac{(kp+m-1)(1-\theta)}{\rho}} (M\varepsilon)^{kp+m}\mathcal{B}\left(1-\frac{d}{2 sa},1-\sigma(kp+m-1)\theta\right),
\end{eqnarray}
where we have used the fact that $t_s\geq t$ for all $t>0$, and where $\mathcal{B}$ denotes the beta function, under the following assumptions:
$$\frac{d}{2 sa}<1,\qquad \sigma(kp+m-1)\theta<1,\quad\hbox{and}\quad 1-\frac{d}{2 sa}-\sigma(kp+m-1)\theta=0.$$
It remains to establish the existence of the parameters $\theta=\theta_k$, $\rho=\rho_k$, $k\geq0$, and $a$. Since $q>(m-1)p/(p-1)$, we can select the  parameters as follows. First we choose
$$\frac{1-\frac{d(p-1)}{2 sp}}{\sigma(pk+m-1)}<\theta_k<\frac{1}{pk+m-1}\min(m-1,\frac{1}{\sigma})=\frac{m-1}{pk+m-1},$$
since $\sigma=1/(m-1)-d/(2s q)<1/(m-1)$. The lower bound is imposed to ensure compatibility with the condition $r>1$. Next, we construct $\rho_k$.  We require the relation $1-d/(2 sa)-\sigma(kp+m-1)\theta_k=0$. Using the identity in \eqref{propA}, we choose $\rho_k$ so that
$$\frac{1-\theta_k}{\rho_k}=\frac{2 s}{d(kp+m-1)}-\frac{2 s\theta_k}{d(m-1)}.$$
This in particular yields $\rho_k\geq d(m-1)/(2 s)\geq p$. Finally, we fix $a>0$ so that the remaining balance condition in \eqref{propA} is satisfied, thereby completing the selection of parameters. With these choices, using the identity $\mathcal{B}(x,y)=\Gamma(x)\Gamma(y)/\Gamma(x+y)$, for all $x,y>0$, we obtain
\begin{equation}\label{eq16}
\mathcal{B}\left(1-\frac{d}{2sa},1-\sigma(kp+m-1)\theta\right)=\Gamma\left(1-\frac{d}{2 sa}\right)\Gamma\left(\frac{d}{2 sa}\right)\leq C,
\end{equation}
We also observe that $\theta_k\longrightarrow 0$ and $\rho_k\longrightarrow \infty$ as $k\rightarrow\infty,$ which implies
$$\frac{(kp+m-1)(1-\theta_k)}{p\rho_k}(1+\rho_k)\leq k, \quad \hbox{for all}\,\, k\geq1.$$
This, together with the estimate $\Gamma(x+1)\leq C\,x^{x+1/2}$ for all $x\geq 1$, yields
\begin{equation}\label{eq17}
\left(\Gamma\left(\frac{\rho_k}{p}+1\right)\right)^{\frac{(kp+m-1)(1-\theta_k)}{\rho_k}} \leq C^k\,k!.
\end{equation}
Combining (\ref{eq15}), (\ref{eq16}) and (\ref{eq17}), we obtain
$$
\left\|\int_0^te^{-(t-\tau)\mathcal{L}} f(u(\tau))\,d\tau\right\|_{p} \leq C\sum_{k=0}^{\infty}(C\,\lambda)^k (M\varepsilon)^{kp+m}\leq C(M\varepsilon)^m,
$$
for $\varepsilon$ sufficiently small. This concludes the proof of \eqref{eq113}. We now turn to the proof of \eqref{eq14}. Using the relation $s=d(p-1)/(2 p)$ along with Lemma \ref{lemma5}, and $p'=p/(p-1)$, we deduce
 $$\left\|\int_0^te^{-(t-\tau)\mathcal{L}} f(u(\tau))\,d\tau\right\|_{L^\infty(0,\infty;L^\phi)}\leq C\|f(u(\tau))\|_{L^\infty(0,\infty;L^1\cap L^{2p}\cap L^{2p'})}.$$ 
Arguing as in the case of $s<n(p-1)/(2 p)$ (see (\ref{eq19})), we deduce $\left\|f(u(t))\right\|_{r}\leq C(M\varepsilon)^m$, for $r=1,2p, 2p'$, and all $t>0$. This completes the proof of \eqref{eq14}.\\
To estimate the second term of $\Phi(u)$ in $\exp L^p(\mathbb{R}^d)$ in the case $s>n(p-1)/(2 p)$, let $b>0$ be the positive number satisfying $b=2\ln(b+1)$. In fact, $2<b<3$ and
$$2\ln(z+1)-z\geq0,\qquad \hbox{for all}\,\, 0\leq z\leq b.$$
 Consequently, one verifies that
 \begin{equation}\label{eq18}
\left(\ln\left((t-\tau)_s^{-d/2s}+1\right)\right)^{-1/p}\leq 2^{1/p}(t-\tau)_s^{d/2s p},\quad\hbox{for all}\,\,0\leq \tau\leq t-b^{-2s/d}.
\end{equation}
If $t\leq b^{-2s/d}$, then by Proposition \ref{techprop1}, for any $r\geq 1$, we obtain
$$
\left\|\int_0^te^{-(t-\tau)\mathcal{L}} f(u(\tau))\,d\tau\right\|_{\exp L^p} \leq  C\int_0^t\left((t-\tau)_s^{-\frac{d}{2 s r}}+1\right)\left(\|f(u(\tau))\|_r+\|f(u(\tau))\|_1\right)\,d\tau.
$$
Taking $r=p'=p/(p-1)$ and using the assumption $s>d(p-1)/(2p)$, together with the inequality $(t-\tau)_s\geq (t-\tau)$, it follows that
\begin{equation}\label{eq20}
\left\|\int_0^t e^{-(t-\tau)\mathcal{L}} f(u(\tau))\,d\tau\right\|_{\exp L^p} \leq C\|f(u)\|_{L^\infty(0,\infty;L^1\cap L^{p'})}.
\end{equation}
Next, since $m>p$, and as  in \eqref{eq19}, one obtains $\|f(u(t))\|_{r}
\leq  C (M\varepsilon)^m$, for  $r=1,p'$, and all $t>0$. Therefore,
\[
\left\|\int_0^t e^{-(t-\tau)\mathcal{L}} f(u(\tau))\,d\tau\right\|_{\exp L^p}
=
O(\varepsilon).
\]
For $t> b^{-2s/d}$, we split the integral into two parts
\begin{eqnarray*}
\left\|\int_0^t e^{-(t-\tau)\mathcal{L}} f(u(\tau))\,d\tau\right\|_{\exp L^p}
&\le&
\int_0^{t-b^{-2s/d}}
\left\|e^{-(t-\tau)\mathcal{L}} f(u(\tau))\right\|_{\exp L^p}
\,d\tau
\\
&&
+\int_{t-b^{-2s/d}}^t
\left\|e^{-(t-\tau)\mathcal{L}} f(u(\tau))\right\|_{\exp L^p}
\,d\tau
\\
&=:& I + J.
\end{eqnarray*}
In the same spirit as in  \eqref{eq20}, where the integration is performed over $[0,b^{-s/2d}]$, we argue similarly on the interval $[t-b^{-s/2d},t]$ to obtain $J \leq C(M\varepsilon)^m.$\\
 On the other hand, using Proposition \ref{techprop1}-(2) together with \eqref{eq18}, we estimate $I$ as follows:
$$
I \leq
C \int_0^{t}
(t-\tau)_s^{-\frac{d}{2s}\left(\frac{1}{a}-\frac{1}{p}\right)}
\|f(u(\tau))\|_{a}\,d\tau,
$$
where $1\leq a\leq p$,  Applying the same argument used to derive \eqref{eq113} (under the same assumptions on $s$ and $m$), we deduce that
\[
I = O(\varepsilon).
\]
Combining the estimates for $I$ and $J$, we conclude that, for $t>b^{-2s/d}$,
\[
\left\|\int_0^t e^{-(t-\tau)\mathcal{L}} f(u(\tau))\,d\tau
\right\|_{\exp L^p}
= O(\varepsilon).
\]
Since the above bound is uniform in time, it follows that
\[
\left\|\int_0^t e^{-(t-\tau)\mathcal{L}} f(u(\tau))\,d\tau
\right\|_{L^\infty(0,\infty;\exp L^p)}
= O(\varepsilon),
\qquad \text{for all } t>0.
\]
It remains to show that
\[
t_s^\sigma
\left\|
\int_0^t e^{-(t-\tau)\mathcal{L}} f(u(\tau))\,d\tau
\right\|_{q}
= O(\varepsilon),
\qquad \text{for every } t>0,
\]
which guarantees that $\Phi(u)\in \mathcal{U}_\varepsilon$. This follows by an argument analogous to \eqref{eq21}, together with the fact that $f(0)=0$.\\
We now proceed to prove that $\Phi$ is a contraction on $\mathcal{U}_{\varepsilon}$. Let $u,v\in \mathcal{U}_\varepsilon$. By \eqref{pqbound}, we obtain
\begin{eqnarray*}
\|\Phi(u)-\Phi(v)\|_{q}
&\leq&
C\int_0^t
(t-\tau)_s^{-\frac{d}{2 s}\left(\frac{1}{r}-\frac{1}{q}\right)}
\|f(u(\tau))-f(v(\tau))\|_{r}
\,d\tau,
\end{eqnarray*}
for every $1\leq r\leq q$. Moreover, by invoking assumption \eqref{eq4a} together with the power-series expansion of the exponential function, we obtain the pointwise estimate
$$
|f(u)-f(v)|
\le C \sum_{k=0}^{\infty}
\frac{\lambda^k}{k!}
|u-v|
\left(
|u|^{kp+m-1}
+
|v|^{kp+m-1}
\right).
$$
Applying H\"older's inequality then yields
$$
\|f(u)-f(v)\|_{r}
\leq C\sum_{k=0}^{\infty}\frac{\lambda^k}{k!}
\|u-v\|_{q}
\left(
\|u\|_{{a(kp+m-1)}}^{kp+m-1}
+
\|v\|_{{a(kp+m-1)}}^{kp+m-1}
\right).
$$
Next, we use H\"older interpolation to further estimate the norms, which gives
\begin{eqnarray*}
\|f(u)-f(v)\|_{r}
&\leq&
C\sum_{k=0}^{\infty}\frac{\lambda^k}{k!}
\|u-v\|_{q}
\Big(
\|u\|_{q}^{(kp+m-1)\theta}
\|u\|_{\rho}^{(kp+m-1)(1-\theta)}
\\
&&\qquad\qquad\qquad\qquad\quad
+
\|v\|_{q}^{(kp+m-1)\theta}
\|v\|_{\rho}^{(kp+m-1)(1-\theta)}
\Big),
\end{eqnarray*}
where the parameters satisfy
\[
\frac{1}{a(kp+m-1)}
=
\frac{\theta}{q}
+
\frac{1-\theta}{\rho},\qquad \frac{1}{r}
=
\frac{1}{q}
+
\frac{1}{a},
\]
with $0\leq \theta\leq 1$. Consequently, using Lemma \ref{lemma1}-(3) and assuming that $p\leq \rho<\infty$, we infer that
\begin{eqnarray*}
\|\Phi(u)-\Phi(v)\|_{q}
&\leq&
C\sum_{k=0}^{\infty}\frac{\lambda^k}{k!}
\left(\Gamma\!\left(\frac{\rho}{p}+1\right)\right)^{\frac{(kp+m-1)(1-\theta)}{\rho}}
\\
&&\times
\int_0^t
(t-\tau)_s^{-\frac{d}{2 sa}}
\tau_s^{-\sigma(1+(kp+m-1)\theta)}
\\
&&\times
\Big(
(\tau_s^\sigma\|u\|_{q})^{(kp+m-1)\theta}
\|u\|_{\exp L^p}^{(kp+m-1)(1-\theta)}
\\
&&\qquad\qquad
+
(\tau_s^\sigma\|v\|_{q})^{(kp+m-1)\theta}
\|v\|_{\exp L^p}^{(kp+m-1)(1-\theta)}
\Big)
\,d\tau.
\end{eqnarray*}
Since $u,v\in \mathcal{U}_\varepsilon$, this further yields
\begin{eqnarray}\label{estA}
\|\Phi(u)-\Phi(v)\|_{q}&\leq&
C\, d(u,v)\, (\varepsilon M)^{m-1}
\sum_{k=0}^{\infty}
(\varepsilon M)^{kp}
\frac{\lambda^k}{k!}\left(\Gamma\!\left(\frac{\rho}{p}+1\right)\right)^{\frac{(kp+m-1)(1-\theta)}{\rho}}\nonumber
\\
&&\times \int_0^t
(t-\tau)_s^{-\frac{d}{2 sa}}
\tau_s^{-\sigma(1+(kp+m-1)\theta)}
\,d\tau.
\end{eqnarray}
To estimate the last integral term on the right-hand side of \eqref{estA}, we use a Beta-function type argument. In particular, we have
$$
\int_0^t
(t-\tau)_s^{-\alpha}
\tau_s^{-\beta}
\,d\tau
\leq
t_s^{-\sigma}
\mathcal{B}(1-\alpha,1-\beta),
$$
where $\alpha=d/(2sa)$ and $\beta=\sigma(1+(kp+m-1)\theta)$, provided that
$1-\alpha-\beta=-\sigma$ and $\alpha,\beta<1$.
Indeed, if $t\geq1$, then
$$\int_0^t
(t-\tau)_s^{-\alpha}
\tau_s^{-\beta}
\,d\tau
\leq\int_0^t
(t-\tau)^{-\alpha}
\tau^{-\beta}
\,d\tau=
t^{-\sigma}
\mathcal{B}(1-\alpha,1-\beta)=t_s^{-\sigma}
\mathcal{B}(1-\alpha,1-\beta).
$$
If $t\leq 1$, we similarly obtain
$$\int_0^t
(t-\tau)_s^{-\alpha}
\tau_s^{-\beta}
\,d\tau
=\int_0^t
(t-\tau)^{-s\alpha}
\tau^{-s\beta}
\,d\tau=
t^{1-s\alpha-s\beta}
\mathcal{B}(1-s\alpha,1-s\beta)\leq t_s^{-\sigma}
\mathcal{B}(1-\alpha,1-\beta).
$$
Therefore, we conclude that
\begin{eqnarray*}
t_s^\sigma
\|\Phi(u)-\Phi(v)\|_{q}
&\leq&
C\, d(u,v)\, (\varepsilon M)^{m-1}
\sum_{k=0}^{\infty}
(\varepsilon M)^{kp}
\frac{\lambda^k}{k!}
\\
&&\times
\left(\Gamma\!\left(\frac{\rho}{p}+1\right)\right)^{\frac{(kp+m-1)(1-\theta)}{\rho}}\mathcal{B}(1-\alpha,1-\beta),
\end{eqnarray*}
As in the previous argument, for each $k\geq 0$ we first fix a parameter $\theta=\theta_k\geq 0$ such that
\[
\frac{1-\frac{d(q-1)}{2sq}}{\sigma(pk+m-1)}
<
\theta_k
<
\frac{1}{pk+m-1}
\min\left\{m-1,\frac{1-\sigma}{\sigma}\right\},
\]
where we have also used that $q>(m-1)p/(p-1)\geq m.$
Next, we choose $\rho=\rho_k$ satisfying
\[
\frac{1-\theta_k}{\rho_k}
=
\frac{2s}{d(kp+m-1)}
-
\frac{2s\theta_k}{d(m-1)},
\]
and we finally select $a>0$ such that
\[
\frac{1}{a(kp+m-1)}
=
\frac{\theta_k}{q}
+
\frac{1-\theta_k}{\rho_k}.
\]
In order to guarantee that $\sigma<1$, we further assume $
q<d(m-1)/(2s(2-m)_+)$, where $(\cdot)_+$ denotes the positive part.
With these choices of parameters, and arguing as in the previous case, the Beta-function term remains uniformly bounded, namely $\mathcal{B}\!\left(
1-\alpha,
\,1-\beta
\right)=
\Gamma\!\left(1-\alpha\right)
\Gamma\!\left(1-\beta\right)/
\Gamma\!\left(2-\alpha-\beta\right)
\leq C$, and moreover,
\[
\left(
\Gamma\!\left(\frac{\rho_k}{p}+1\right)
\right)^{\frac{(kp+m-1)(1-\theta_k)}{\rho_k}}
\leq
C^k\,k!.
\]
As a result, we arrive at
\begin{equation}\label{eq21}
t_s^\sigma
\|\Phi(u)-\Phi(v)\|_{q}
\leq
C\, d(u,v)\,(M\varepsilon)^{m-1}
\sum_{k=0}^{\infty}
(C\lambda)^k (M\varepsilon)^{kp}.
\end{equation}
By choosing $\varepsilon>0$ sufficiently small, it follows that
\[
t_s^\sigma
\|\Phi(u)-\Phi(v)\|_{q}
\leq
\frac{1}{2}\, d(u,v),
\]
which shows that $\Phi$ is a contraction on $\mathcal{U}_\varepsilon$.
This completes the proof of the existence of a global solution in Theorem \ref{theo2} in the case
$s\geq d(p-1)/(2p)$. The estimate \eqref{decay1.8} follows directly from the fact that
$u\in \mathcal{U}_\varepsilon$.

We now turn to the proof of \eqref{eq5}, namely the continuity of $u$ at $t=0$ in $\exp L^p(\mathbb{R}^d)$.
From the mild formulation, we have
$$
\|u(t)-e^{-t\mathcal{L}}u_0\|_{\exp L^p}
\le
\int_0^t
\|e^{-(t-\tau)\mathcal{L}}f(u(\tau))\|_{\exp L^p}
\,d\tau .
$$
Applying Lemma \ref{lemma1}-(1), we obtain
$$
\|u(t)-e^{-t\mathcal{L}}u_0\|_{\exp L^p}\lesssim
\int_0^t
\|e^{-(t-\tau)\mathcal{L}}f(u(\tau))\|_{p}\,d\tau
+\int_0^t
\|e^{-(t-\tau)\mathcal{L}}f(u(\tau))\|_{\infty}\,d\tau.
$$
Let $q>\max\{d/(2s),1\}$. 
Using \eqref{pqbound}, this yields
\begin{align}
\|u(t)-e^{-t\mathcal{L}}u_0\|_{\exp L^p}
&\lesssim\int_0^t
\|f(u(\tau))\|_{p}\,d\tau+\int_0^t
(t-\tau)_s^{-\frac{d}{2sq}}
\|f(u(\tau))\|_{q}\,d\tau .
\label{p3}
\end{align}
For $r=p$ or $r=q$, using
$
|f(u)|
\le C|u|^m e^{\lambda |u|^p}$ together with H\"older's inequality, we infer
\[
\|f(u)\|_{r}
\le
C \|u\|_{2mr}^m
\|e^{\lambda |u|^p}-1\|_{2r}
+ C \|u\|_{mr}^m.
\]
Since $2mr\ge mr\ge m\ge p$, Lemma \ref{lemma1}-(3) ensures that
\[
\|u\|_{2mr}+\|u\|_{mr}
\le C \|u\|_{\exp L^p}.
\]
Combining this with Lemma \ref{lemma6} and the fact that $u\in \mathcal{U}_\varepsilon$, we deduce
\begin{equation}
\|f(u)\|_{r}
\le C\|u\|_{\exp L^p}^m.
\label{p4}
\end{equation}
Substituting \eqref{p4} into \eqref{p3} and using that $u\in L^\infty(0,\infty;\exp L^p(\mathbb{R}^d))$, we obtain
\[
\|u(t)-e^{-t\mathcal{L}}u_0\|_{\exp L^p}
\le
C t \|u\|_{L^\infty(0,\infty;\exp L^p)}^m
+
C t^{1-\frac{d}{2sq}}
\|u\|_{L^\infty(0,\infty;\exp L^p)}^m,
\]
which completes the proof of continuity at $t=0$.

To conclude the proof of Theorem \ref{theo2}, it remains to verify that $u(t) \rightarrow u_0$ 
in $\exp L^p(\mathbb{R}^d)$ as $ t \to 0,$ in the weak$^{*}$ sense. Let $X = L^1(\ln L)^{1/p}(\mathbb{R}^d)$ be the pre-dual space of 
$\exp L^p(\mathbb{R}^d)$. 
It is well known that $X$ is a Banach space and that $C_0^\infty(\mathbb{R}^d)$ is dense in $X$ 
(see \cite{Adams}). For any $\varphi \in X$, using the duality between $\exp L^p(\mathbb{R}^d)$ and $X$, together with \eqref{heatsemimix} and H\"older's inequality in Orlicz spaces, we obtain
\begin{align*}
\left|
\int_{\mathbb{R}^d}
\bigl(e^{-t\mathcal{L}}u_0(x)-u_0(x)\bigr)\varphi(x)\,dx
\right|
&=
\left|
\int_{\mathbb{R}^d}
u_0(x)\bigl(e^{-t\mathcal{L}}\varphi(x)-\varphi(x)\bigr)\,dx
\right| \\
&\le
2\|u_0\|_{\exp L^p}
\|e^{-t\mathcal{L}}\varphi-\varphi\|_X.
\end{align*}
Since $C_0^\infty(\mathbb{R}^d)$ is dense in $X$ and the semigroup
$e^{-t\mathcal{L}}$ is strongly continuous on $X$ 
(see the argument in the proof of Proposition \ref{prop2}), it follows that
\[
\lim_{t\to 0}
\|e^{-t\mathcal{L}}\varphi-\varphi\|_X=0.
\]
Consequently,
\[
\lim_{t\to 0}
\int_{\mathbb{R}^d}
\bigl(e^{-t\mathcal{L}}u_0-u_0\bigr)\varphi\,dx
=0
\qquad \text{for all } \varphi\in X,
\]
which establishes the weak$^{*}$ convergence. This completes the proof of Theorem  \ref{theo2}.
\hfill $\square$

 
\section*{Acknowledgment}

Dharmendra Kumar Chaurasia acknowledges the financial support provided by the University Grants Commission (UGC), India, under Grant No. F.No. 211610055687, and sincerely thanks his supervisor, Dr. Ashish Pathak, for the support and encouragement.

\end{document}